\documentclass[oneside,11pt]{amsart}

\usepackage{amsmath,amssymb,amsthm, mathrsfs, mathtools}
\usepackage{graphicx}
\usepackage{amscd,mathtools}
\usepackage{enumitem} \setlist{nosep} 
\usepackage[utf8]{inputenc}
\usepackage[T1]{fontenc}
\usepackage[english]{babel}
\usepackage{color}
\usepackage[pagebackref,breaklinks,hidelinks,unicode]{hyperref}
\usepackage{float}
\usepackage{comment}
\usepackage{mathtools}
\usepackage{csquotes}
\usepackage[font=small,justification=centering]{caption}
\usepackage{subcaption}
\usepackage[dvipsnames]{xcolor}
\usepackage{mathabx} 
\usepackage{xfrac} 
\usepackage[normalem]{ulem} 
\usepackage{dirtytalk}

\usepackage[margin=2.98cm]{geometry}

\usepackage{tikz}
\usepackage{tikz-cd}
\usetikzlibrary{automata}
\usetikzlibrary{calc,decorations.pathreplacing}
\usetikzlibrary{arrows}
\usetikzlibrary{decorations.pathreplacing}
\usetikzlibrary{intersections}

\makeatletter
    \newtheorem*{rep@theorem}{\rep@title}
    \newcommand{\newreptheorem}[2]{%
    \newenvironment{rep#1}[1]{%
    \def\rep@title{#2 \ref{##1}}%
    \begin{rep@theorem}}%
    {\end{rep@theorem}}}
\makeatother

\newtheorem{theorem}{Theorem}[section]
\newtheorem{proposition}[theorem]{Proposition}
\newreptheorem{proposition}{Proposition}
\newtheorem{corollary}[theorem]{Corollary}
\newtheorem{lemma}[theorem]{Lemma}

\newtheorem{Set up}[theorem]{Set-up}
\newtheorem{question}[theorem]{Question}

\theoremstyle{definition}

\newtheorem{definition}[theorem]{Definition}

\newtheorem{example}[theorem]{Example}

\newtheorem{remark}[theorem]{Remark}

\newtheorem*{answer*}{Answer}
\newtheorem*{application*}{Application}

\DeclarePairedDelimiterX{\Norm}[1]{\lVert}{\rVert}{#1}
\theoremstyle{definition}

\renewcommand*{\frak}{\mathfrak}

  \newcommand{\calA}{\mathcal{A}}
  
  \newcommand{\calC}{\mathcal{C}}

  \newcommand{\calI}{\mathcal{I}}

  \newcommand{\calV}{\mathcal{V}}
  \newcommand{\calW}{\mathcal{W}}

\renewcommand*{\backrefalt}[4]{\ifcase #1 (Not cited).\or (Cited p.~#2).\else (Cited pp.~#2).\fi} 

\newcounter{shcount}
\newcommand*{\bsh}[1]{\theoremstyle{definition}\newtheorem{subhead\theshcount}[theorem]{#1}
    \begin{subhead\theshcount}} 
\newcommand*{\esh}{\end{subhead\theshcount}\stepcounter{shcount}} 

\newcounter{enumlabelcount}
\makeatletter\newcommand\enumlabel[1][]{\item[#1]
    \refstepcounter{enumlabelcount}\def\@currentlabel{#1}}\makeatother

\definecolor{harrycomment}{rgb}{0.6,0,0.4}

\DeclareMathOperator{\dist}{\mathsf{d}}

\DeclareMathOperator{\diam}{diam}
\DeclareMathOperator{\vbig}{Big}

\DeclareMathOperator{\cay}{Cay}

\DeclareMathOperator{\isom}{Isom}
\DeclareMathOperator{\mcg}{MCG}
\DeclareMathOperator{\Out}{Out}

\DeclareMathOperator{\Stab}{Stab}

\newcommand*{\R}{\mathbb R}
\newcommand*{\s}{\mathfrak S}

\newcommand*{\Z}{\mathbb Z}

\newcommand{\eps}{\varepsilon}
\renewcommand*{\epsilon}{\eps}
\newcommand*{\mc}{\mathcal}

\newcommand*{\nest}{\sqsubseteq}
\newcommand*{\pnest}{\sqsubsetneq}

\newcommand*{\pconest}{\sqsupsetneq}
\newcommand*{\trans}{\pitchfork}

\DeclareMathOperator{\relevant}{Rel}
\DeclareMathOperator{\Rel}{Rel}

\newcounter{claimcount}

\newenvironment{claim*}[1]{\par\vspace{2mm}\noindent
    \underline{Claim:}\hspace{2mm}#1}{}
\newenvironment{claimproof}[1]{\par\vspace{2mm}\noindent\underline{Proof:}\hspace{2mm}#1}
    {\leavevmode\unskip\penalty9999\hbox{}\nobreak\hfill\quad\hbox{$\diamondsuit$}\vspace{2mm}}

\makeatletter\def\subsection{\@startsection{subsection}{1}\z@{.7\linespacing\@plus\linespacing}
    {.5\linespacing}{\normalfont\scshape\centering}}\makeatother 

\title[Translation lengths in HHGs]{Uniform undistortion from barycentres, and applications to hierarchically hyperbolic groups} 
\hypersetup{pdftitle = Translation lengths}
\author{Carolyn Abbott}
\address{Department of Mathematics, Brandeis University, Waltham, MA USA}
\email{carolynabbott@brandeis.edu}
\author{Mark Hagen}
\address{School of Mathematics, University of Bristol, Bristol, United Kingdom}
\email{markfhagen@posteo.net}
\author{Harry Petyt}
\address{Mathematical Institute, University of Oxford, UK}
\email{petyt@maths.ox.ac.uk}
\author{Abdul Zalloum}

\address{Institute for Advanced Study in Mathematics, Harbin Institute of Technology, China}
\email{abdul.zalloum@utoronto.ca}

\begin{document}

\begin{abstract}
We show that infinite cyclic subgroups of groups acting uniformly properly on injective metric spaces are uniformly undistorted. In the special case of hierarchically hyperbolic groups, we use this to study translation lengths for actions on the associated hyperbolic spaces. We then use quasimorphisms to produce examples where these latter results are sharp.
\end{abstract}

\maketitle

\setcounter{tocdepth}{2}\tableofcontents\setcounter{tocdepth}{3}

\section{Introduction}

An informative object relating to an isometric action of a group $G$ on a metric space $(X,\dist)$ is the \emph{(stable) translation length} function $\tau_X\colon G\to[0,\infty)$, which is given by 
$$\tau_X(g)=\lim_{n\to\infty}\frac{\dist(x,g^nx)}{n}.$$
(It is independent of $x$ by the triangle inequality.)  In a classical situation, $X$ is the universal cover of a Riemannian manifold $M$, with $G=\pi_1M$, and $\tau_X(g)$ corresponds to the length of a geodesic in $M$ representing the conjugacy class of $g$.  However, translation lengths are much more broadly applicable, for instance in work of Gromoll--Wolf on actions on nonpositively-curved manifolds \cite{gromollwolf:some} and work of Gersten--Short on biautomaticity \cite{gerstenshort:rational}.

While the precise values of the function $\tau_X$ depend on both $X$ and the action, many properties of the image of $\tau_X$ (the translation length \emph{spectrum}) do not.  For instance, if $G$ acts isometrically on $X$ and $Y$ and they admit a $G$--equivariant $(K,K)$--quasi-isometry, then $\frac{1}{K}\tau_X\leq \tau_Y\leq K\tau_X$. In particular, $g\in G$ is loxodromic with respect to $X$ (i.e., $\tau_X(g)>0$) if and only if it is loxodromic with respect to $Y$.

A $G$--action on $X$ is \emph{translation discrete} if there is a constant $\tau_0>0$ such that for each $g\in G$, either $\langle g\rangle$ has bounded orbits in $X$ (in particular, $\tau_X(g)=0$), or $\tau_X(g)\geq \tau_0$. The above shows that translation discreteness is invariant under equivariant quasi-isometries. In this paper, we consider translation discreteness in two main settings: for $X$ a proper Cayley graph of $G$; and for (improper) actions on hyperbolic spaces arising in the context of hierarchically hyperbolic groups.

\subsection{Translation discreteness in Cayley graphs} 

Let $G$ be a finitely generated group with a word metric. In this case, if $\tau_G(g)>0$, then $g$ has infinite order and the inclusion of the cyclic subgroup it generates is a quasiisometric embedding: $\langle g\rangle$, or $g$, is \emph{undistorted}. Translation discreteness of the action of $G$ on itself therefore equates to infinite cyclic subgroups being \emph{uniformly} undistorted.

Uniform undistortion has stronger consequences than mere undistortion. For instance, it implies that solvable subgroups of finite virtual cohomological dimension are abelian \cite{conner:discreteness}, and finitely generated abelian subgroups are undistorted (though not necessarily uniformly) \cite{button:aspects}. 

\bsh{Non-examples} \label{rem:conner}
The simplest obstruction to uniform undistortion is the existence of distorted elements. Well-known examples include Baumslag--Solitar groups $BS(p,q)$ with $|p|\neq|q|$, and virtually nilpotent groups that are not virtually abelian \cite[Lem.~14.15]{drutukapovich:geometric}. Despite having distorted elements, in many such examples there is a translation length \emph{gap}: there exists $\tau_0>0$ such that if $\tau_G(g)\ne0$, then $\tau_G(g)\geq\tau_0$. 

For example, let $G$ be the integer Heisenberg group $\langle x,y,z\mid [x,z]=[y,z]=1,[x,y]=z\rangle$. All powers of $z$ have zero translation length. However, every other $g\in G$ has nontrivial image under the $1$--Lipschitz epimorphism $G\to\mathbb Z^2$ given by $x\mapsto(1,0)$, $y\mapsto(0,1)$, $z\mapsto (0,0)$, and hence has $\tau_G(g)\ge1$. The Baumslag-Solitar group $H=\langle a\rangle*_{a^b=a^2}$ is similar: $\tau_H(a)=0$, but by considering axes in the Bass--Serre tree of the given splitting, one sees that $\tau_H$ is bounded away from $0$ on $H-\langle a\rangle$.

There also exist soluble groups $G$ where $\tau_G$ takes arbitrarily small values, even though $\tau_G(g)>0$ for all $g\ne1$ \cite{conner:discreteness}. Indeed, Conner gives an explicit linear map $M\in\operatorname{GL}_4(\mathbb Z)$ such that all cyclic subgroups of the group $G=\mathbb Z^4\rtimes_M\mathbb Z$ are undistorted, but not uniformly \cite[Eg.~7.1]{conner:discreteness}.
\esh

Uniform undistortion is known in various nonpositively curved situations, such as in hyperbolic groups \cite{gromov:hyperbolic,swenson:hyperbolic,delzant:sousgroupes}, CAT(0) groups \cite{conner:translation}, Garside groups \cite{huangosajda:helly}, Helly groups \cite{haettelosajda:locally}, groups satisfying various (graphical) small-cancellation conditions \cite{ArzhantsevaCashenGruberHume:negative}, mapping class groups \cite{bowditch:tight}, and $\Out (F_n)$ \cite{alibegovic:translation}.  It is often established directly by constructing uniform-quality quasi-axes in some space on which $G$ acts. A beautiful example is Haglund's construction of combinatorial axes for loxodromic isometries of CAT(0) cube complexes \cite{haglund:isometries}, which generalises the classical case of trees \cite{serre:trees}. 

Our first theorem extends this list to cover injective spaces. See Definition~\ref{defn:injective-space} for injective spaces, and Definition~\ref{defn:acyl} for uniform properness. Note that every proper and cobounded action is uniformly proper.

\begin{theorem}\label{thm:translation discrete-injective}
Let $G$ be a group acting uniformly properly on an injective metric space $X$. Infinite-order elements of $G$ are uniformly undistorted. Moreover, every infinite-order $g\in G$ admits an $\eps$--quasi-axis in $X$ for every $\eps>0$.
\end{theorem}

 Theorem~\ref{thm:translation discrete-injective} is a combination of Propositions~\ref{prop:undistorted} and~\ref{prop:quasiaxes}. 

Note that Theorem~\ref{thm:translation discrete-injective} does not require $X$ to be proper or the action to be cocompact. Under those assumptions, the statement about quasi-axes can be strengthened: each infinite-order $g\in G$ is semisimple by \cite[Thm~II.6.10]{bridsonhaefliger:metric}, hence has positive translation length by \cite[Prop.~1.2]{lang:injective}, and so has a geodesic axis. In particular, translation discreteness was already known for groups acting properly and cocompactly on proper injective spaces.  This includes Helly groups \cite[Thm~6.3]{chalopinchepoigenevoishiraiosajda:helly} and cocompactly cubulated groups \cite{bowditch:median}. 

The greater generality of Theorem~\ref{thm:translation discrete-injective} can already be seen in \emph{hierarchically hyperbolic groups}, which act properly and \emph{coboundedly} on injective spaces \cite{haettelhodapetyt:coarse,petytzalloum:constructing}. For many such groups, such as mapping class groups, we expect that there is no proper cocompact action on an injective space. Hierarchical hyperbolicity will be discussed further in Section~\ref{subsec:intro_HHG}. 


\begin{corollary}[HHGs are translation discrete]\label{cor:HHG-translation discrete}
Let $(G,\mathfrak S)$ be a hierarchically hyperbolic group. There exists $\tau_0>0$ such that $\tau_G(g)\geq \tau_0$ for all infinite-order $g\in G$.
\end{corollary}

Corollary~\ref{cor:HHG-translation discrete} strengthens earlier results of \cite{durhamhagensisto:boundaries,durhamhagensisto:correction}, by showing that HHGs have uniformly undistorted cyclic subgroups; those earlier results did not establish uniformity. Together with \cite[Thm~4.2]{button:aspects}, one recovers the fact that their abelian subgroups are undistorted, previously established in \cite[Cor.~H]{haettelhodapetyt:coarse} and \cite[Prop.~2.17]{hagenrussellsistospriano:equivariant}.

\subsection{Non-proper actions arising from hierarchical hyperbolicity} \label{subsec:intro_HHG} 

The second main setting in which we consider translation discreteness is for certain non-proper actions on hyperbolic spaces that arise in the study of hierarchical hyperbolicity.

There are many examples illustrating the importance of non-proper actions on hyperbolic spaces, with the actions of mapping class groups on curve graphs being perhaps the most prominent. Other well-known examples include the extension graph of a right-angled Artin group \cite{kimkoberda:embedability,kimkoberda:geometry}, and the contact graph of a cubulated group \cite{hagen:weak} more generally.

For the action of the mapping class group on the curve graph, translation discreteness was established by Bowditch as a consequence of acylindricity \cite{bowditch:tight}, and since then there has been considerable study of the translation length spectrum for mapping class groups and analogous examples, e.g., \cite{baikseoshin:onfiniteness,aougabtaylor:pseudoanosovs,bellwebb:polynomialtime,mangahas:recipe,genevois:translation}. The situation for cubulated groups is more complicated: although the action on the contact graph is WPD \cite{behrstockhagensisto:hierarchically:1}, it need not be acylindrical \cite{shepherd:cubulation}, and translation lengths can even accumulate on zero \cite{genevois:translation}. These issues disappear when the cube complex admits a \emph{factor system}, however \cite{behrstockhagensisto:hierarchically:1}. This is a very general situation \cite{hagensusse:onhierarchical} that includes the case where the group is virtually compact special \cite{haglundwise:special}.

Introduced in \cite{behrstockhagensisto:hierarchically:1,behrstockhagensisto:hierarchically:2}, hierarchically hyperbolic groups (HHGs) are a common generalisation of mapping class groups and virtually compact special groups that now encompasses a wide range of examples \cite{behrstockhagensisto:hierarchically:2,berlairobbio:refined,robbiospriano:hierarchical,berlynerussell:hierarchical,chesser:stable,behrstockhagenmartinsisto:combinatorial,hagenmartinsisto:extra,hagenrussellsistospriano:equivariant}. More background will be given in Section~\ref{sec:background}, but for now let us just say that part of the definition of $G$ being an HHG includes a set $\s$, on which $G$ acts cofinitely, and a hyperbolic space $\calC U$ associated with each element $U\in\s$. Moreover, $G$ acts by isometries on the extended metric space $\prod_{U\in\s}\calC U$, permuting the factors. We are interested in the action of $\Stab_G(U)$ on $\mathcal CU$.

In \cite{durhamhagensisto:boundaries,durhamhagensisto:correction}, it was shown that (up to passing to a uniform power), for each infinite-order $g\in G$, there is some $U\in\s$ stabilised by $g$ and such that the translation length $\tau_{\mathcal CU}(g)$ is positive.  Undistortion of cyclic subgroups in HHGs is a straightforward consequence. However, the argument cannot be adapted to yield a uniform lower bound on $\tau_{\calC U}(g)$ or $\tau_G(g)$.  Our logic in this paper is reversed: we use the lower bound on $\tau_G(g)$ provided by Corollary~\ref{cor:HHG-translation discrete} to analyse $\tau_{\mathcal CU}(g)$. More precisely, we prove the following.

\begin{theorem}\label{thm:main}
Let $(G,\s)$ be a hierarchically hyperbolic group.  There exist $n \in \mathbb Z$ and $\tau_0>0$, depending only on $G$ and $\s$, such that for any infinite-order $g\in G$ there exists $U\in\s$ with $\tau_{\calC U}(g^n)\geq \tau_0$.
\end{theorem}

Theorem~\ref{thm:main} does not assert that the action of any particular $\Stab_G(U)$ on $\mathcal CU$ is translation discrete, and indeed one cannot hope for a general statement of this type, as illustrated by the following result. See Section~\ref{sec:strong_18} for definitions.

\begin{theorem}\label{thm:intro-counterexamples}
Let $G$ be an infinite hierarchically hyperbolic group that is either elementary or acylindrically hyperbolic, and let $[\alpha]\in H^2(G,\mathbb Z)$ be representable by a bounded cocycle. The corresponding $\mathbb Z$--central extension $E_\alpha$ of $G$ admits a hierarchically hyperbolic structure $(E_\alpha,\mathfrak S)$ such that the following holds for some $E_\alpha$-invariant $A\in\mathfrak S$. For all $\epsilon>0$, there exists $g\in E_\alpha$ such that $\tau_{\calC A}(g)\in(0,\epsilon)$.
\end{theorem}

The theorem applies, for example, to any central extension of any infinite hyperbolic group $G$, since $H^2_b(G,\mathbb Z)\to H^2(G,\mathbb Z)$ is onto in that case \cite{mineyev:bounded}.  Even $\Z^2$ admits hierarchically hyperbolic structures involving arbitrarily small translation lengths; see Example~\ref{exmp:modified_abelian}. These examples show that Theorem~\ref{thm:main} is sharp.

Theorem~\ref{thm:main} still gives useful control of translation lengths.  For example, it plays a role in recent work of Abbott--Ng--Spriano and Gupta--Petyt on uniform exponential growth in HHGs and cubulated groups with factor systems \cite{abbottngsprianoguptapetyt:hierarchically}.  It is also used in work of Zalloum \cite{Zalloum:effective}, who shows that in any acylindrically hyperbolic, virtually torsion-free HHG, one can find a Morse element (and more generally, a free stable subgroup) uniformly quickly. (Note that not all HHGs are virtually torsionfree \cite{hughes:lattices}.) It is additionally related to a subtle point about \emph{product regions} and their coarse factors, see \cite[\S2]{durhamhagensisto:correction} and \cite[\S15]{casalsruizhagenkazachkov:real}.

\subsection{Outline of the paper}\label{subsec:outline} 

Section~\ref{sec:UU} discusses translation lengths for actions on spaces with barycentres, which extract a key property of injective spaces. Section~\ref{sec:background} covers background on hierarchical hyperbolicity and applies Proposition~\ref{prop:undistorted} to prove uniform undistortion for HHGs. This is the starting point for the proof of Theorem~\ref{thm:main}, which occupies Section~\ref{sec:proof-of-main}. Section~\ref{sec:strong_18} then discusses the sharpness of Theorem~\ref{thm:main} and some well-known groups. Finally, in Section~\ref{sec:questions} we raise some questions related to this work.

\medskip\textbf{Acknowledgements.} We would like to thank Jason Behrstock, Matt Durham, Thomas Ng, Alessandro Sisto,  and Davide Spriano for discussions related to Theorem~\ref{thm:main}.  We are also grateful to Johanna Mangahas, Lee Mosher, Kasra Rafi, Saul Schleimer, Juan Souto, Richard Webb, and MathOverflow user Sam Nead for helpful discussions about translation lengths on annular curve graphs.  We thank David Hume for pointing out the examples in \cite{ArzhantsevaCashenGruberHume:negative}, Alice Kerr for reminding us about \cite{manning:quasiactions}, and Nicolas Vaskou for a discussion about translation lengths in Artin groups.  We are very grateful to the referees for numerous comments that improved the paper.

CA was partially supported by NSF grants DMS-2106906 and DMS-2340341.  HP thanks the organisers of the thematic program in geometric group theory at the CRM in Montreal, where this project was completed.

\section{Barycentres and translation lengths}\label{sec:UU}

\begin{definition}\label{defn:barycentres}
Let us say that a metric space $(X,\dist)$ \emph{has barycentres} if for each $n$ there is a map $b_n\colon X^n\to X$ such that $b_n$ is:
\begin{itemize}
\item   idempotent: for every $x\in X$ we have $b_n(x^n)=x$, where $x^n$ denotes the tuple $(x,\ldots,x)\in X^n$;
\item   symmetric: $b_n$ is invariant under permutation of co-ordinates;
\item   $\isom X$--invariant: for every $g\in\isom X$ we have $gb_n(x_1,\dots,x_n)=b_n(gx_1,\dots,gx_n)$; 
\item   $\frac1n$--Lipschitz: $\dist\big(b_n(x_1,\dots,x_n),b_n(y_1,\dots,y_n)\big)\,\le\,\frac1n\sum_{i=1}^n\dist(x_i,y_i)$.
\end{itemize}
\end{definition}

In fact, it would be natural to assume something \emph{a priori} rather stronger than idempotence, namely that the barycentre of a repeated tuple agrees with the barycentre of the tuple, in the sense of Remark~\ref{rem:barycentre_multiples}. However, we do not explicitly need this in our arguments, and when $X$ is complete it is actually a consequence of the above definition, as explained in the aforementioned remark.

\begin{definition} \label{def:bicombing}
Let $X$ be a metric space. A \emph{bicombing} $\sigma$ on $X$ is a choice of path $\sigma_{xy}=\sigma(x,y)$ from $x$ to $y$ for every $x,y\in X$. 
\begin{itemize}
\item   If every $\sigma(x,y)$ is a geodesic, then $\sigma$ is a \emph{geodesic bicombing}.
\item   If $\sigma(x,y)=\sigma(y,x)$ for all $x,y\in X$, then $\sigma$ is \emph{reversible}.
\item   We say that a geodesic bicombing $\sigma$ is \emph{conical} if the following holds for all $x,y,x',y'\in X$:
\[
\dist(\sigma_{xy}(t),\sigma_{x'y'}(t)) \,\le\, (1-t)\dist(x,x')+t\dist(y,y').
\]
\end{itemize} 
\end{definition}

Building on work of Es-Sahib--Heinich and Navas \cite{essahibheinich:barycentre,navas:L1}, Descombes showed that every complete metric space with a reversible, conical bicombing (and every proper space with a conical bicombing) that is $\isom$--invariant has barycentres \cite[Thm~2.1]{descombes:asymptotic}. This includes many examples of interest, including all \emph{CUB spaces} \cite{haettel:link}, examples of which are produced in \cite{haettelhodapetyt:lp}. Most importantly for the purposes of this paper, it includes all \emph{injective metric spaces} by work of Lang \cite[Prop.~3.8]{lang:injective}, as mentioned in the introduction. However, there is a more direct way to see that injective spaces have barycentres \cite[\S7.2, \S7.4]{petyt:onlarge}, as we now describe. 

\begin{definition}\label{defn:injective-space}
A metric space $X$ is \textit{injective} if for every metric space $B$ and every subset $A\subseteq B$, if $f\colon A\to X$ is 1--Lipschitz then there exists a 1--Lipschitz map $\hat f\colon B\to X$ with $\hat f|_A=f$.
\end{definition}

For example, given a metric space $Y$, let $\R^Y=\{Y\to\R\}$, and for $f,g\in\R^Y$ let $\dist_\infty(f,g)=\sup_{y\in Y}|f(y)-g(y)|$. Every component of the extended metric space $(\R^Y,\dist_\infty)$ is injective. 

\begin{lemma} \label{lem:injective_barycentres}
Injective metric spaces have barycentres.
\end{lemma}

\begin{proof}
Let $(X,\dist)$ be an injective space. The map $x\mapsto\dist(x,\cdot)$ defines an isometric embedding of $X$ into $\R^X$, so we can view $X$ as a subset.  We now recall a construction from \cite[\S1]{dress:trees}.  First, let $P_X\subseteq \R^X$ be the set of functions $f$ such that $f(x)+f(y)\geq \dist(x,y)$ for all $x,y\in X$.  Observe that the map $(\R^X)^n\to\R^X$ defined by $(f_1,\ldots,f_n)\mapsto \frac{1}{n}\sum_{i=1}^nf_i$ sends tuples of functions in $P_X$ to functions in $P_X$.  

Let $T_X\subseteq P_X$ be the set of $f$ such that 
$$f(x)=\sup_{y\in X}\{\dist(x,y)-f(y)\}$$
for all $x\in X$. Observe that $X$, regarded as above as a subset of $\R^X$, is contained in $T_X$. Indeed, for each $z\in X$ we have $\dist(z,x)=\sup_{y\in X}\{\dist(x,y)-\dist(z,y)\}$ by the triangle inequality. On the other hand, as noted in \cite{dress:trees}, injectivity of $X$ implies that $X\to T_X$ is surjective, and we can identify $X$ with $T_X$.  

Each isometry $\Psi\colon X\to X$ extends to a linear isomorphism $f\mapsto f\Psi^{-1}$ which is an isometry on each component of $\R^X$ and which preserves $P_X$ and $T_X$.  Dress defines a $1$--Lipschitz retraction $p\colon P_X\to T_X=X$ in \cite[\S1.9]{dress:trees} which, by construction, is $\isom(X)$--equivariant (see also \cite[Prop. 3.7.(2)]{lang:injective}).  (From Definition~\ref{defn:injective-space}, one could construct a $1$--Lipschitz retraction $\R^X\to X$ directly, but we use the map $p$ to ensure equivariance.)

Given points $x_1,\dots,x_n$ in $X$, consider their affine barycentre $\frac1n\sum_{i=1}^n\dist(x_i,\cdot)\in P_X$. The maps defined by $b_n\colon(x_1,\dots,x_n) \mapsto p\big(\frac1n\sum_{i=1}^n\dist(x_i,\cdot)\big)$ satisfy the requirements of Definition~\ref{defn:barycentres}, and hence provide barycentres for $X$. 
\end{proof}

\begin{lemma} \label{lem:geodesic}
If $X$ has barycentres, then every pair of points is joined by an isometric image of a dense subset of an interval. If $X$ is complete, then it has an $\isom X$--invariant, reversible, conical bicombing; and every space with an $\isom$--invariant conical bicombing has barycentres.
\end{lemma}

In our applications, we only need the first assertion in the above lemma, and have included the statements about bicombings only since they may be of independent interest.  

\begin{proof}[Proof of Lemma~\ref{lem:geodesic}]
Given $x,y\in X$, the point $b_2(x,y)$ has $\dist(x,b_2(x,y))=\dist(y,b_2(x,y))=\frac12\dist(x,y)$. Iterating, we get an isometrically embedded dyadic interval from $x$ to $y$, as in the first assertion. If $X$ is complete, then we get a uniquely defined geodesic from $x$ to $y$ by taking limits of the dyadic interval. By the properties of barycentres, these geodesics form a reversible, $\isom X$--invariant, conical bicombing. The fact that every space with an $\isom$--invariant conical bicombing has barycentres is \cite[Thm~2.1]{descombes:asymptotic}.
\end{proof}

\begin{remark} \label{rem:barycentre_multiples}
If $X$ is complete, then Lemma~\ref{lem:geodesic} implies that it has a reversible conical bicombing. According to \cite[Prop.~2.4]{descombes:asymptotic}, the barycentres can then be perturbed so that they additionally satisfy $b_{nm}(x_1^m,\dots,x_n^m)=b_n(x_1,\dots,x_n)$ for every $n,m,x_1,\dots,x_n$, where $z^m$ denotes the tuple $(z,\dots,z)\in X^m$. We could therefore have assumed this stronger property to begin with in most situations. Moreover, observe that the barycentres on $X$ naturally extend to its metric completion.
\end{remark}

Throughout this paper, we will assume that all actions of a group on a metric space are by isometries.  Recall that for a group $G$ acting on a metric space $(X,\dist)$ and an element $g\in G$, the stable translation length is denoted $\tau_X(g)=\lim_{n\to\infty}\frac1n\dist(x,g^nx)$, which is independent of $x$. We also write $|g|=\inf\{\dist(x,gx)\,:\,x\in X\}$. Observe that, by repeatedly applying the triangle inequality, we always have $\tau_X(g) \leq |g|$. The following was noted for injective spaces in \cite[Rem.~7.25]{petyt:onlarge}. We provide a proof for completeness.

\begin{lemma} \label{lem:stable_unstable}
If $G$ acts on a metric space $X$ with barycentres, then $|g|=\tau_X(g)$ for all $g\in G$.
\end{lemma}

\begin{proof}
Fix $x\in X$, and let $x_n=b_n(x,gx,\dots,g^{n-1}x)$. We compute
\begin{align*}
\dist(x_n,gx_n) \,&=\, \dist\Big(b_n(x,gx,\dots,g^{n-1}x),\,b_n(gx,g^2x,\dots,g^nx)\Big) \\
    &=\, \dist\Big(b_n(x,gx,\dots,g^{n-1}x),\,b_n(g^nx,gx,\dots,g^{n-1}x)\Big) 
    \,\le\, \frac1n\dist(x_n,g^nx_n).
\end{align*}
Hence $\tau_X(g)\le|g|\le\dist(x_n,gx_n)\to\tau_X(g)$.
\end{proof}

\begin{definition}\label{defn:acyl}
Let $G$ be a group acting on a metric space $X$. 
The action is said to be \emph{ proper} if for every $\eps>0$ there exists $N$ such that
\[
|\{g\in G\,:\,\dist(x,gx)\le\eps\}|< \infty
\]
for all $x\in X$.
The action is \emph{uniformly proper} if for every $\eps>0$ there exists $N$ such that
\[
|\{g\in G\,:\,\dist(x,gx)\le\eps\}|\le N
\]
for all $x\in X$.
The action is  \emph{acylindrical} if for every $\eps>0$ there exist $R,N$ such that if $\dist(x,y)>R$, then
\[
|\{g\in G\,:\,\dist(x,gx),\dist(y,gy)\le\eps\}|\le N.
\]
\end{definition}

Uniform properness implies acylindricity, and proper cobounded actions are uniformly proper.  The proof of the first part of the following proposition is slightly simpler than Bowditch's proof for acylindrical actions on hyperbolic graphs \cite{bowditch:tight}, and also recovers it because hyperbolic graphs are coarsely dense in their injective hulls \cite{lang:injective}.

\begin{proposition} \label{prop:undistorted}
Let $G$ act on a metric space $X$ with barycentres. If the action is:
\begin{itemize}
\item   \emph{acylindrical}, then there exists $\delta>0$ such that $\tau_X(g)>\delta$ for every $g\in G$ whose action is not elliptic;
\item   \emph{uniformly proper}, then $\tau_X(g)>\delta$ for every infinite-order $g\in G$;
\item   \emph{proper and cobounded}, then $G$ has finitely many conjugacy classes of finite subgroups.
\end{itemize}
\end{proposition}

\begin{proof}
Supposing that the action is acylindrical, let $R$ and $N$ be such that if $\dist(x,y)>R$ then $|\{g\in G\,:\,\dist(x,gx),\dist(y,gy)\le1\}|\le N$. Suppose that $g\in G$ is not elliptic. By Lemma~\ref{lem:stable_unstable} there is some $x\in X$ such that $\dist(x,gx)\le\tau_X(g)+\frac1{2N}$. As $g$ is not elliptic, there is some $n$ such that $\dist(x,g^nx)>R$. If $\tau_X(g)\le\frac1{2N}$, then
\[
\dist(g^nx,g^{n+i}x) \ =\ \dist(x,g^ix) \ \le\ i\left(\tau_X(g)+\frac1{2N}\right) \ \le\ \frac iN
\]
for all $i\ge0$. Considering $i\in\{0,\dots,N\}$ gives a contradiction, proving the first statement.

If the action is proper, then no infinite-order element is elliptic, and if the action is uniformly proper then it is acylindrical.

Finally, suppose the action is proper and cobounded. Let $x\in X$. If $F=\{1,f_2,\dots,f_n\}$ is a finite subgroup of $G$, then $b_n(x,f_2\cdot x,\dots,f_n\cdot x)$ is fixed by $F$. Using that finite subgroups have fixed points, a standard argument shows that $G$ has finitely many conjugacy classes of finite subgroups; see \cite[I.8.5]{bridsonhaefliger:metric}, for instance.
\end{proof}

For a constant $\eps\ge0$ and an element $g$ of a group acting on a metric space $X$ with barycentres, let $M_\eps(g)=\{x\in X\,:\,\dist(x,gx)\le\tau_X(g)+\eps\}$. If $\eps>0$, then this set is nonempty by Lemma~\ref{lem:stable_unstable}. Given a proper cocompact action, $M_0(g)\neq\emptyset$ as well, see \cite[Lem.~4.3]{descombeslang:flats}.

\begin{lemma} \label{lem:minsets}
$M_\eps(g)$ is closed under taking barycentres and is setwise stabilised by the action of the centraliser of~$g$.
\end{lemma}

\begin{proof}
If $\{x_1,\dots,x_n\}\subseteq M_\eps(g)$, then
\begin{align*}
\dist\Big(b_n(x_1,\dots,x_n),gb_n(x_1,\dots,x_n)\Big) 
    \,&=\, \dist\Big(b_n(x_1,\dots,x_n),b_n(gx_1,\dots,gx_n)\Big) \\
&\le\, \frac1n\sum_{i=1}^n\dist(x_i,gx_i) \,\le\, \tau_X(g)+\eps.
\end{align*}
If $h$ commutes with $g$ and $x\in M_\eps(g)$, then $\dist(hx,ghx)=\dist(hx,hgx)\le\tau_X(g)+\eps$.
\end{proof}

\begin{definition}
For $q\ge0$, a $q$--\emph{quasiaxis} of an isometry $g$ of a metric space $X$ is a $\langle g\rangle$--invariant subset $A_g\subseteq X$ admitting a $q$--coarsely surjective $(1+q,q)$--quasi-isometry $\R\to A_g$.
\end{definition}

\begin{proposition}\label{prop:quasiaxes}
Let $G$ act on a metric space $X$ with barycentres. For every $\eps>0$, every $g\in G$ with $\tau_X(g)>0$ has a $\eps$--quasiaxis.  

Hence, if $G$ acts on $X$ acylindrically, then for all $\epsilon> 0$, every non-elliptic $g\in G$ has an $\epsilon$--quasiaxis.  If the action is uniformly proper, then this holds for all $g\in G$ of infinite order.
\end{proposition}

\begin{proof}
As noted in Remark~\ref{rem:barycentre_multiples}, the barycentre maps on $X$ naturally extend to its completion, so there is no loss in assuming that $X$ is complete. Let $g\in G$ satisfy $\tau\coloneqq\tau_X(g)>0$. Let $\eps>0$. We are free to assume that $\eps<\tau$. Let $x\in M_\eps(g)$. Let $I\colon [0,\dist(x,gx)]\to X$ be the unit-speed geodesic from $x$ to $gx$ provided by the proof of Lemma~\ref{lem:geodesic}. Note that a dense subset of $I$ is constructed by repeatedly taking barycentres, so Lemma~\ref{lem:minsets} implies that $I\subset M_\eps(g)$. Hence the subset $A_g=\bigcup_{n\in\Z}g^nI$, which is stabilised by $g$, is also contained in $M_\eps(g)$. It remains to show that $A_g$ is an $\eps$--quasiline.

Given $t\in\R$, write $t=n\tau+r$, where $n\in\Z$ and $r\in[0,\tau)$, and let $f(t)=g^nI(r)$, which is well-defined since $\tau\leq |I|$. This defines a map $f:\R\to A_g$ that is $\eps$--coarsely onto. Let $t_1,t_2\in\R$, and write $t_i=n_i\tau+r_i$ for $i\in\{1,2\}$. If $n_1=n_2$, then by definition we have $\dist(f(t_1),f(t_2))=|t_1-t_2|$. Otherwise, we may assume that $n_1<n_2$, and we compute:
\begin{align*}
\dist(f(t_1),f(t_2)) \,&\le\, \dist(f(t_1),g^{n_1+1}x)+\dist(g^{n_1+1}x,g^{n_2}x)+\dist(g^{n_2}x,f(t_2)) \\
    &\le\, (\tau+\eps-r_1)+(n_2-n_1-1)(\tau+\eps)+r_2 \\
    &=\, r_2-r_1+\tau(n_2-n_1)+\eps(n_2-n_1) \\
    &=\, (t_2-t_1) + \frac\eps\tau((t_2-r_2)-(t_1-r_1)) \\
    &\le\,  |t_2-t_1|+\eps|t_2-t_1|+\eps.
\end{align*}
We similarly obtain a lower bound as follows:
\begin{align*}
\dist(f(t_1),f(t_2)) \,&\ge\, \dist(g^{n_1}x,g^{n_2+1}x)-\dist(g^{n_1}x,f(t_1))-\dist(g^{n_2+1}x,f(t_2)) \\
    &\ge\, (n_2-n_1+1)\tau-r_1-(\tau+\eps-r_2) \\
    &=\, \tau(n_2-n_1)+r_2-r_1-\eps \,=\, |t_2-t_1|-\eps.
\end{align*}
Combining these estimates, we see that $f$ is a $(1+\eps,\eps)$--quasi-isometry. The statement about acylindrical and uniformly proper actions now follows using Proposition~\ref{prop:undistorted}.
\end{proof}

By Lemma~\ref{lem:injective_barycentres}, injective spaces have barycentres in the sense of Definition~\ref{defn:barycentres}.

The same proof shows that if $X$ is complete and $g$ is non-elliptic with $M_0(g)\neq\emptyset$, then $g$ has a geodesic axis. However, in our applications we will not be able to arrange for $M_0$ to be nonempty, because $X$ can fail to be proper. 

We finish this section by partially addressing Question~\ref{question:flats-barycentres}, which asks for a higher-dimensional version of Proposition~\ref{prop:quasiaxes}.

\begin{proposition} \label{prop:uniformly_flat_tori}
Let $G$ act properly coboundedly on a metric space $X$ with barycentres, and let $\eps>0$.  If $H=\langle g_1,\dots,g_n\rangle\cong\Z^n$ is a free abelian subgroup of $G$,  then there is an $H$--equivariant $(1+\eps,\eps)$--coarsely Lipschitz map $(\R^n,\ell^1)\to X$.
\end{proposition}

For $n=1$, the above proposition gives a $\langle g_1\rangle$--equivariant uniformly coarsely Lipschitz axis in $X$, but does not recover the full statement of Proposition~\ref{prop:quasiaxes} because it gives only one of the bounds needed for a quasi-isometric embedding.

\begin{proof}[Proof of Prop.~\ref{prop:uniformly_flat_tori}]
Let $\delta\in(0,1]$ be given by Proposition~\ref{prop:undistorted}, and fix $\eps>0$. 

We first show that $\bigcap_{i=1}^nM_\eps(g_i)\neq\emptyset$, arguing by induction on $n$.  By Lemma~\ref{lem:stable_unstable}, $M_\eps(g_1)\neq\emptyset$.  Fix $j\in\{2,\ldots,n\}$, and assume by induction that there exists $x\in\bigcap_{i=1}^{j-1}M_\eps(g_i)$. By Lemma~\ref{lem:minsets}, for every $m$, the point $y_m=b_m(x,g_jx,\dots,g_j^{m-1}x)$ lies in $\bigcap_{i=1}^{j-1}M_{\eps}(g_i)$, using that $[g_j,g_i]=1$ for all $i$. Moreover, 
\[
\dist(y_m,g_jy_m) \,=\, \dist\Big(b_m(x,g_jx,\dots,g_j^{m-1}x),b_m(g_j^mx,g_jx,\dots,g_j^{m-1}x)\Big) 
\,\le\, \frac1m\dist(x,g_j^mx).
\]
Thus $y_m\in\bigcap_{i=1}^jM_{\eps}(g_i)$ for sufficiently large $m$. So $\bigcap_{i=1}^nM_{\eps}(g_i)\neq\emptyset$. Fix $x\in\bigcap_{i=1}^nM_{\eps}(g_i)$.  

Next, let $\{e_1,\dots,e_n\}$ be the standard basis of $\R^n$.  For brevity, let $\tau_i=\tau_X(g_i)$ for $1\leq i\leq n$.  Let $D^n\subseteq \R^n$ be the set of vectors of the form $\sum_{i=1}^nr_i\tau_ie_i$, where $r_i$ is a dyadic rational.  Define a map $f\colon D^n\to X$ as follows.

Set $f(0)=x$. For $i\ge1$, suppose that $f$ has been defined on $\left(D^{i-1}\times\{0\}^{n-i+1}\right)\cap\prod_{s=1}^n[0,\tau_s)$. Given an element $p$ thereof, set $f(p+\frac12\tau_ie_i)=b_2(f(p),g_if(p))$. We can define $f(q)$ for every $q\in D^i\times\{0\}^{n-i}\cap\prod_{s=1}^n[0,\tau_s)$ by repeatedly taking barycentres in this way. Inductively, this defines $f$ on $D^n\cap\prod_{s=1}^n[0,\tau_s)$.  Specifically, for any $p\in D^{n-1}\times\{0\}$, the restriction of $f$ to the set of dyadic rationals in $\{p\}\times[0,\tau_n)$ has image an isometrically embedded (dense subset of an) interval of length $\dist(f(p),g_nf(p))$ from $f(p)$ to $g_nf(p)$.

Given $p\in D^n$, we can write
\[
p=(p_i\tau_i + a_i\tau_i)_{i=1}^n,
\]
where each $p_i\in[0,1)$ is a dyadic rational and each $a_i\in\Z$.  We define $[p]=(p_i\tau_i)_{i=1}^n$.  Then we let $f(p)=g_1^{a_1}\cdots g_n^{a_n}f([p])$.  Letting $H$ act on $D^n$ by declaring $g_i$ to be a unit translation by $\tau_ie_i$, observe that $f$ is well-defined and $H$--equivariant by construction.  

We now check that $f$ is coarsely Lipschitz.   For convenience, let $C=D^n\cap\prod_{i=1}^n[0,\tau_i)$, and let $\bar C$ be its closure in $D^n$.  Note that taking this closure \textit{in $D^n$} has the effect of adding the endpoint of $\tau_i$ if it is a dyadic rational, and leaving the interval half-open otherwise.   Note that $C$ contains exactly one point in each $H$--orbit.

We first show that $f$ is $(1+\epsilon/\delta)$--Lipschitz on $\bar C$. If $p,q\in \bar C$, we can write $p=\sum_{i=1}^np_i\tau_ie_i$ and $q=\sum_{i=1}^nq_i\tau_ie_i$.  By construction, 
\begin{eqnarray*}
 \dist(f(p),f(q))&\leq&\sum_{i=1}^n(\tau_i+\eps)|p_i-q_i|=\sum_{i=1}^n(1+\frac{\eps}{\tau_i})\tau_i|p_i-q_i|\\
 &\leq&(1+\frac{\eps}{\delta})\sum_{i=1}^n\tau_i|p_i-q_i|=(1+\frac{\eps}{\delta})\|p-q\|_1.
\end{eqnarray*}
Here we used that $\tau_i\geq \delta>0$ for all $i$.

We next show that  $f$ is $(1+\epsilon/\delta)$--Lipschitz on $D^n$.  Let $p,q\in D^n$ be given.  Let $\gamma$ be an $\ell^1$--metric geodesic in $\R^n$ from $p$ to $q$ such that $\gamma\cap D^n$ is dense in $\gamma$.  Then $\gamma=\gamma_1\cdots\gamma_m$, where each $\gamma_j$ is an $\ell^1$--metric geodesic whose intersection with $D^n$ lies in some $H$--translate of $\bar C$.  Since $f$ is $(1+\eps/\delta)$--Lipschitz on $\bar C$, it follows from  equivariance that $f$ is $(1+\eps/\delta)$--Lipschitz on each $H$--translate of $\bar C$.  Hence, letting $p=p_0,\ldots,p_{m-1}$ be the initial points of $\gamma_1,\ldots,\gamma_m$ and $p_{m}=q$ the terminal point of $\gamma_m$, we have $\dist(f(p_i),f(p_{i+1}))\leq (1+\eps/\delta)\|p_i-p_{i+1}\|_1$ for each $i$. We conclude that
\[\dist(f(p),f(q))\leq\sum_{i=0}^{m-1}\dist(f(p_i),f(p_{i+1}))\leq(1+\eps/\delta)\sum_i|\gamma_i|=(1+\eps/\delta)\|p-q\|_1,\]  
as required. 

Finally, we note that since $D^n$ is dense in $\R^n$, we can extend $f\colon D^n\to X$ equivariantly to a map $f\colon \R^n \to X$ that is $(1+\epsilon/\delta,\epsilon/\delta)$--coarsely Lipschitz.  Since this argument holds for all $\epsilon$, it holds, in particular, for $\epsilon \delta$, which concludes the proof.
%
\end{proof}

One problem with Proposition~\ref{prop:uniformly_flat_tori} is that the map $f$ could, \emph{a priori}, fail to be colipschitz. This is addressed by the following statement, which is similar to \cite[Thm~4.2]{button:aspects}.

\begin{lemma} \label{lem:undistorted_abelian}
Let $H=\langle g_1,\dots,g_n\rangle\cong\Z^n$ be a free abelian group acting on a metric space $X$. For every $T,\delta>0$ there exists $\delta'=\delta'(n,\delta,T)>0$ such that the following holds. If every $h\in H$ has $\tau_X(h)>\delta$ and $\max\{\tau_X(g_i)\}\le T$, then in fact every $h\in H$ has $\tau_X(h)\ge\delta'\dist_H(1,h)$. 
\end{lemma}

\begin{proof}
In the terminology of \cite[Thm~4.2]{button:aspects}, $\tau_X$ defines a $\Z$--norm on $H$. Consider the group embedding $H\to\R^n$ given by $g_i\mapsto e_i$, where the $e_i$ are the standard basis vectors. The $\Z$--norm $\tau_X$ extends to a norm $N$ on $\R^n$. By linearity, $N(x)\ge r\|x\|_1$ for all $x\in\R^n$, where $r=\inf\{N(z)\,:\,\|z\|_1=1\}$. Let us find a lower bound for $r$.

Let $x\in\R^n$ have $\|x\|_1=1$. By an application of the pigeonhole principle, there is a constant $M=M(n,\delta,T)$, a natural number $q\le M$, and integers $p_1,\dots,p_n$ such that $|qx_i-p_i|\le\frac1{2n}\frac\delta T$ (see, e.g., \cite[Thm~201]{hardywright:introduction}). Following \cite{steprans:characterization}, let $p=(p_1,\dots,p_n)$, so that $\|qx-p\|_1\le\frac12\frac\delta T$. Because $N$ is a norm, every point $z$ with $\|z\|_1=1$ has $N(z)\le\max\{N(e_i)\}\le T$, which shows that $N(qx-p)\le\frac\delta2$. As $\|qx\|_1\ge\|x\|_1=1$, the vector $p$ must be nonzero, hence $N(p)>\delta$, and therefore $N(qx)>\frac\delta2$. We have shown that $r>\frac\delta{2M}$.
If $h\in H$, then $\|h\|_1=\dist_H(1,h)$, so we can compute
\[
\tau_X(h) \,=\, N(h) \,\ge\, r\|h\|_1 \,>\, \frac\delta{2M}\dist_H(1,h). \qedhere
\]
\end{proof}

Proposition~\ref{prop:uniformly_flat_tori} generalises a result of Descombes--Lang for proper spaces with \emph{convex, consistent} geodesic bicombings \cite[Thm~1.2]{descombeslang:flats}, which includes proper injective spaces of finite dimension by \cite{descombeslang:convex}. They prove that if $G$ acts properly cocompactly on such a space $X$ in such a way that the bicombing is $G$--invariant, then every free-abelian subgroup $A<G$ of rank $n$ acts by translations on some subset $Y\subseteq X$ isometric to $(\R^n,N)$, where $N$ is some norm. Although $Y$ is very well controlled, it does not seem clear whether this implies that abelian subgroups of $G$ are uniformly undistorted, because the norm $N$ depends on the choice of $A$.

\section{Background on hierarchical hyperbolicity}\label{sec:background}

A \emph{hierarchically hyperbolic structure} on a space $(X,\dist)$ is a package of associated data, which is usually abbreviated to $(X,\s)$. Despite the compact notation, this package holds a  large amount of information, much of which is not directly relevant to our purposes here (though it is all \emph{indirectly} relevant, via the ``distance formula'' below). We therefore summarise the main components of the definition and some basic results needed here. For the detailed definition, see \cite[\S1]{behrstockhagensisto:hierarchically:2}; for a mostly self-contained exposition of the theory, see \cite[Part~2]{casalsruizhagenkazachkov:real}.

Firstly, $\s$ denotes the \emph{index set}, whose elements are called \emph{domains}. In some of the following statements, we refer to a constant $E\geq 1$, which is part of the data of a hierarchically hyperbolic structure and which is fixed in advance. In particular, in any property of individual domains $V\in\mathfrak S$, the constant $E$ is independent of $V$.

\medskip

\begin{enumerate}
\item   For each domain $W \in \mathfrak S$ there is an associated $E$-hyperbolic geodesic space $\calC W$ and an $E$--coarsely surjective $(E,E)$-coarsely Lipschitz map $\pi_W\colon X \rightarrow \calC W$.

\item   $\s$ has  mutually exclusive relations $\nest$, $\bot$, and $\trans$ satisfying the following.
    \begin{itemize}
    \item   $\nest$ is a partial order called \emph{nesting}. If $\mathfrak S \neq \emptyset$, then $\mathfrak S$ contains a unique $\nest$-maximal element $S$.
    \item   $\bot$ is a symmetric and anti-reflexive relation called \emph{orthogonality}.  If $U\nest V$ and $V\perp W$, then $U\perp W$.

    \item   There exists an integer $c$ called the \emph{complexity} of $X$ such that every $\pnest$--chain has length at most $c$, and every pairwise orthogonal set has cardinality at most $c$.
    \item   $\trans$, called \emph{transversality}, is the complement of $\bot$ and $\nest$.     
        
    \end{itemize}

\item   If $U\sqsubsetneq V$ or $U\pitchfork V$, then there is an associated set $\rho^U_V\subseteq \calC V$ of diameter at most $E$. If $U\pnest V\pnest W$, then $\dist_{\calC W}(\rho^U_W,\rho^V_W)\le E$.
        
\item   If $U \pitchfork V$ and $x \in X$ satisfies $\dist_{\calC U}(\pi_U(x), \rho^V_U)> E$, then $\dist_{\calC V}(\pi_V(x), \rho^U_V)\leq E$.  
\end{enumerate} 

We emphasise that the above list is just a subset of the full definition of a hierarchically hyperbolic structure.

\bsh{Convention}
When $A$ and $B$ are subsets of a metric space $X$, we write $\dist(A,B)=\inf\{\dist(a,b)\,:\,a\in A,\,b\in B\}$.
Note that this does not define a metric on the set of diameter--$\leq E$ sets in $X$ since there is an error of $2E$ in the triangle inequality.  This explains the mysterious appearances of extra multiples of $E$ in our later computations.
\esh

\begin{definition}[Hierarchically hyperbolic group] \label{defn:equivariance} 
A finitely generated group $G$ with word metric $\dist=\dist_G$ is a \emph{hierarchically hyperbolic group} (or \emph{HHG}) if it has a hierarchically hyperbolic structure $(G,\mathfrak S)$ such that the following additional equivariance conditions hold.
\begin{itemize}
\item   $G$ acts on $\mathfrak S$. The action is cofinite and preserves the three relations $\nest$, $\perp$, and $\trans$.
\item   For each $g \in G$ and each $U \in \mathfrak S$, there is an isometry $g\colon \calC U \rightarrow \calC gU$. These isometries satisfy $g\circ h=gh.$
\item   For all $x,g\in G$ and $U\in\mathfrak S$, we have $g \pi_U(x)=\pi_{gU}(gx)$.  Moreover, if $V\in\mathfrak S$ and either $U\trans V$ or $V\pnest U$, then $g \rho^V_U=\rho^{gV}_{gU}$.
\end{itemize}
\end{definition}

We are following the definition given in \cite{petytspriano:unbounded} as it appears to be the most compact, but the notion was originally introduced in \cite{behrstockhagensisto:hierarchically:1,behrstockhagensisto:hierarchically:2}. The original definition was shown to be equivalent to the present, simpler, one in \cite[\S2]{durhamhagensisto:correction}.

Part (4) of a hierarchically hyperbolic structure $(X,\mathfrak S)$ mentioned above is called a \emph{consistency} condition.  A related part of the definition is a ``bounded geodesic image'' axiom, and though we do not use it directly, it combines with consistency to provide the following statement, which will be important for us. It is part of \cite[Prop.~1.11]{behrstockhagensisto:hierarchically:2}. For two points $x,y \in X$, it is standard to simplify notation by using $\dist_U(x,y)$ to denote $\dist_{\calC U}(\pi_U(x), \pi_U(y)),$ and similarly for subsets of $X$. 
\begin{lemma}[Bounded geodesic image]\label{lem:BGIA} 
Let $x,y\in X$, and suppose that $U,V\in\s$ satisfy $V\pnest U$. If there exists a geodesic $\gamma\subseteq \calC U$ from $\pi_U(x)$ to $\pi_U(y)$ such that $\dist_U(\rho^V_U,\gamma)>E$, then $\dist_V(x,y)\le E$.
\end{lemma}

\begin{definition}[Relevant domains]\label{defn:relevant_domains}
Let $D\geq 0$ and let $x,y\in X$.  Then $\relevant_D(x,y)$ denotes the collection of all $U\in\s$ with $\dist_U(x,y) \geq D.$    
\end{definition}

Another axiom from \cite[Def.~1.1]{behrstockhagensisto:hierarchically:2} is the ``large link'' axiom, which we also will not use directly, but instead use via the following consequence:

\begin{lemma}[Passing-up Lemma, {\cite[Lem.~2.5]{behrstockhagensisto:hierarchically:2}}] \label{lem:passingup}
For every $C>0$ there is an integer $P(C)$ such that the following holds. Let $U\in \s$ and let $x,y\in X$. If there is a set $\{V_1,\dots, V_{P(C)}\}$ with $V_i\pnest U$ and $\dist_{V_i}(x,y)>E$ for all $i$, then there exists some domain $W\nest U$ such that $V_i\pnest W$ for some $i$ and $\dist_W(x,y)>C$.
\end{lemma}

One of the most important features of a hierarchically hyperbolic structure is that one has a ``distance formula'' \cite[Thm~4.5]{behrstockhagensisto:hierarchically:2}, which allows one to approximate distances in $X$ using projections to the domains.

\begin{theorem}[Distance formula]\label{thm:DF}
Let $(X,\s)$ be an HHS. There exists $D_0\ge6E$, depending only on the HHS structure, such that the following holds. For every $D\ge D_0$ there exists $A_D$ such that for all $x,y\in X$ we have
\[
\frac1{A_D}\dist_X(x,y)-A_D \,\leq\, \sum_{U\in\relevant_D(x,y)}\dist_U(x,y) \,\leq\, A_D\dist_X(x,y)+A_D.
\]
Moreover, the dependence of $A_D$ on $D$ is entirely determined by the HHG structure.
\end{theorem}

The axioms in \cite[Def.~1.1]{behrstockhagensisto:hierarchically:2} were chosen largely to enable one to prove Theorem~\ref{thm:DF}. From now on, we will work in the context of HHGs, and hence switch notation from $(X,\mathfrak S)$ to $(G,\mathfrak S)$. We are interested in infinite cyclic subgroups of the HHG $(G,\mathfrak S)$ and how they act on the HHG structure. Accordingly, we recall the following definition from \cite[\S6.1]{durhamhagensisto:boundaries}.

\begin{definition}[Bigsets]\label{defn:bigset}
Let $(G,\mathfrak S)$ be an HHG.  For each $g\in G$, let $\vbig(g)$ be the set of domains $U\in\mathfrak S$ such that $\diam\pi_U(\langle g\rangle)=\infty$.
\end{definition}

For an element $g$ of an HHG $(G,\s)$, the set $\vbig(g)$ is empty if and only if $g$ has finite order \cite[Prop.~6.4]{durhamhagensisto:boundaries}. The following properties are established in \cite[\S6]{durhamhagensisto:boundaries}.

\begin{lemma} \label{lem:Big_lemma} 
Let $(G, \mathfrak S)$ be an HHG. Given $g \in G$, write $\vbig(g)=\{U_i\}_{i\in I}$. 
\begin{enumerate}
\item   $gU_i\in\vbig(g)$ for all $i$.
    
\item \label{item:big-orthogonal} $U_i\perp U_j$ for all $i\ne j$. In particular, $|I|\le c$, where $c$ is the complexity of $\mathfrak S$.
    
\item   For all $i\in I$, we have $g^{c!}U_i=U_i$, and so $\langle g^{c!}\rangle$ acts on each $\calC U_i$ by isometries.

\item \label{item:big-bounded-outside} There exists $D=D(g,\mathfrak S)$ such that $\diam\pi_V(\langle g\rangle)\leq D$ for all $V\not\in\vbig(g)$.
\end{enumerate}

\end{lemma}

\begin{remark}\label{rem:eyries}
Many of the statements in Lemma~\ref{lem:Big_lemma} hold when $\langle g\rangle$ is replaced by more complicated subgroups of $G$ --- see \cite[\S9]{durhamhagensisto:boundaries} and \cite{petytspriano:unbounded} --- but we will not use this here.    
\end{remark}

We will use the following proposition, which is \cite[Thm~3.1]{durhamhagensisto:correction}: 

\begin{proposition}\label{prop:big-loxodromic}
If $g$ is an infinite-order element of an HHG $(G,\s)$, of complexity $c$, then $g^{c!}$ acts loxodromically on $\calC U$ for all $U\in\vbig(g)$. In particular, $\tau_G(g)>0$.
\end{proposition}

The assertion about $\tau_G(g)$ follows since $\pi_U$ is coarsely Lipschitz and $\langle g^{c!}\rangle$--equivariant.

The proof of Proposition~\ref{prop:big-loxodromic} given in \cite{durhamhagensisto:correction} relies in an essential way on the constants $D(g,\mathfrak S)$ from Lemma~\ref{lem:Big_lemma}.\eqref{item:big-bounded-outside} and cannot be adapted to give a lower bound on either $\tau_U(g^{c!})$ or $\tau_G(g)$ that is independent of $g$. Indeed, we shall see in Section~\ref{sec:strong_18} that there need not be a uniform lower bound on $\tau_U(g^{c!})$ that holds for \emph{all} $U\in\vbig(g)$. On the other hand, the following proposition states that $\tau_G(g)$ \emph{can} be uniformly lower-bounded. This fact, which relies on the results of Section~\ref{sec:UU}, is an important ingredient in establishing Theorem~\ref{thm:main}.

\begin{proposition}[Uniform undistortion in HHGs]\label{prop:HHG-uniform-undistortion}
Let $(G,\mathfrak S)$ be an HHG.  There exists $\tau_0>0$ such that $\tau_G(g)\geq\tau_0$ for every infinite-order $g\in G$.  Hence there exists $K=K(G,\mathfrak S)$ such that for all infinite-order $g\in G$ and all $x\in G$, we have $\dist_G(x,g^nx)>Kn$ for all $n\geq0$.
\end{proposition}

\begin{proof}
By \cite[Cor.~3.8, Lem.~3.10]{haettelhodapetyt:coarse}, there is a proper, cobounded action of $G$ on an injective metric space $X$.  Fix a basepoint $x_0\in X$ and a constant $\mu\geq 1$ such that the orbit map $G\to X$ given by $h\mapsto hx_0$ is a $(\mu,\mu)$--quasi-isometry.

By Lemma~\ref{lem:injective_barycentres}, $X$ has barycentres.  Since the action of $G$ on $X$ is proper and cobounded, it is uniformly proper.  Hence Proposition~\ref{prop:undistorted} provides a constant $\delta>0$ such that $\tau_X(g)\geq \delta$ for all infinite-order $g\in G$.  A computation shows $\tau_G(g)\geq\frac\delta\mu$.  Recalling that $\tau_G(g^n)\leq \dist(x,g^nx)$ for all $n\geq 0$ and $x\in G$, and that $\tau_G(g^n)=n\tau_G(g)$, we have $\dist_G(x,g^nx)\geq n\frac\delta\mu$.  Taking $K=\frac\delta{2\mu}$ completes the proof. 
\end{proof}

\section{Proof of Theorem~\ref{thm:main}}\label{sec:proof-of-main}

Let $(G,\s)$ be a hierarchically hyperbolic group and $g\in G$ an infinite order element, with $\vbig(g)=\{U_1,\dots, U_m\}$. By Lemma~\ref{lem:Big_lemma}, replacing $g$ by $g^{c!}$, we can and shall assume that $gU_i=U_i$ for all $i$. Independently of $g$, we  bound  $\tau_{U_i}(g)$ below for some $i$.

Our strategy is as follows.  First, we carefully construct a uniform quality quasi-axis for $g$ in each $U_i$ and a point $x\in G$ whose projection to each $\mathcal C U_i$ lies on this quasi-axis.  We next show that the terms in the distance formula for $d_G(x,g^nx)$ can be divided into two sets: the domains that are orthogonal to all $U_i$ and the domains that nest into some $U_i$.   The first technical step is to give an upper bound to the contribution to $d_G(x,g^nx)$ from domains that are orthogonal to all $U_i$.  This gives a lower bound on the contribution from domains that nest into some $U_i$.  The second technical step in the proof uses the passing-up lemma and a counting argument to show that, in fact, some $U_i$ itself must have a uniformly large contribution to the distance formula.  This will then give a uniform lower bound on the translation length $\tau_{U_i}(g)$.  Because the dependence of the constants at each step is crucial to our arguments, we describe every step in detail to make this explicit.

\subsection{Step 1: Quasi-axes}\label{subsec:quasi-axes-CU} 

For each $i\leq m$, Proposition~\ref{prop:big-loxodromic} says that $g$ acts on $\mathcal CU_i$ as a loxodromic isometry. A standard fact of hyperbolic spaces is that every loxodromic isometry has a quasiaxis. We make this precise with the following.

\begin{lemma}\label{lem:quasi-axis}
There is a constant $R$ such that the following hold. Let $k\ge1$ be such that $\calC U$ is $k$--hyperbolic for all $U\in\s$. There exists $\alpha_i\subseteq \mathcal CU_i$ such that:
\begin{itemize}
\item $\alpha_i$, with the subspace metric inherited from $\mathcal CU_i$, is $(Rk,Rk)$--quasi-isometric to $\mathbb R$;
\item $\alpha_i$ is $Rk$--quasiconvex; and
\item $\alpha_i$ is $\langle g\rangle$--invariant.
\end{itemize}
\end{lemma}

\begin{proof}
Since uniformly hyperbolic geodesic spaces are uniformly coarsely dense in their injective hulls \cite[Prop.~1.3]{lang:injective}, this s a consequence of Proposition~\ref{prop:quasiaxes}.
%
%
\end{proof}

\begin{remark} \label{rem:bigger_E}
Since $R$ is a universal constant, there is no harm in increasing $E$ to assume that $E\ge Rk$. Thus, when we later refer to Lemma~\ref{lem:quasi-axis}, we shall take the constants in its conclusion to all be $E$. Actually, we shall later make one final increase of $E$ by an amount dependent only on the partial realisation axiom; see Section~\ref{subsec:which-point}.
\end{remark}

\begin{corollary} \label{cor:translation_bound}
Let $x\in\alpha_i$. If $n>0$ is such that $\dist_{U_i}(x,g^nx)\geq 14E$, then $\tau_{U_i}(g)\ge\frac En$.
\end{corollary}

\begin{proof}
Since $\alpha_i$ is $\langle g\rangle$--invariant and $(E,E)$--quasi-isometric to $\R$, any two points on $\alpha_i$ are moved the same distance by $g^n$, up to an error of at most $5E$. As $\alpha_i$ is $E$--quasiconvex, we can consider the $\langle g\rangle$--equivariant coarse closest point projection $\mathcal CU_i\to\alpha_i$. Given $y\in\mathcal C U_i$, its projection $\bar y$ is $2E$--close to a geodesic from $y$ to $g^ny$ and similarly for $g^n\bar y$. It follows that
\[
\dist_{U_i}(x,g^nx) \,\le\, \dist_{U_i}(\bar y,g^n\bar y)+5E\,\le\, \dist_{U_i}(y,g^ny)+10E.
\]
According to \cite[Prop.~1.3]{lang:injective}, $\calC U_i$ is $E$--coarsely dense in its injective hull $H$, which is $E$--hyperbolic. Lemma~\ref{lem:stable_unstable} shows that there is some $y'\in H$ such that $\dist_H(y',g^ny')\le\tau_{U_i}(g^n)+E$. Choosing $y\in\calC U_i$ so that $\dist_H(y,y')\le E$, we see that $\dist_{U_i}(x,g^nx)\le n\tau_{U_i}(g)+13E$. In particular, if $\dist_{U_i}(x,g^nx)\ge14E$, then $\tau_{U_i}(g)\ge\frac En$.
\end{proof}
 
In view of this corollary, our task is to produce a uniform constant $J$, independent of $g$, such that $\dist_{U_i}(x,g^Jx)>14E$ for some $i$.

\subsection{Step 2: Choosing which point to move}\label{subsec:which-point} 

We fix, for the remainder of the proof, a point $x\in G$ as follows.  For each $i\in\vbig(g)$, fix some $x_i\in\alpha_i$.  Since the elements of $\vbig(g)$ are pairwise orthogonal by Lemma~\ref{lem:Big_lemma}, the partial realisation axiom \cite[Def.~1.1.(8)]{behrstockhagensisto:hierarchically:2} provides a point $x\in G$ such that
\begin{itemize}
\item   $\dist_{U_i}(x,x_i)\leq E$ for all $i$, and
\item   $\dist_V(x,\rho^{U_i}_V)\leq E$ for all pairs $(i,V)$ where either $U_i\pnest V$ or $U_i\trans V$. 
\end{itemize}

With one final uniform enlargement of $E$, for convenience only, we replace each $\alpha_i$ by its $E$--neighbourhood in $\mathcal CU_i$, so that, for this fixed $x\in G$, we have $\pi_{U_i}(x)\in\alpha_i$ for all $i$.

\subsection{Step 3: Organising distance formula terms}\label{subsec:organising-terms} 

Recall that the distance formula, Theorem~\ref{thm:DF} yields a constant $D_0\geq 6E$. We partition the $D_0$--relevant domains as follows.
Fix $n\geq 0$, and let
\[
\calW^n=\{ W \in \Rel_{D_0}(x,g^nx) \,:\, W\perp U_i\text{ for all }i\}
\]
and 
\[
\calV^n_i=\{ V \in \Rel_{D_0}(x,g^nx) \,:\, V \sqsubseteq U_i \},
\]
where $i\in\{1,\dots,m\}$. We denote the union of the $\calV^n_i$ by $\calV^n$. 

Note that $\calV^n_i\cap\calV^n_j=\emptyset$ for $i\neq j$, as $U_i\perp U_j$.  Similarly, $\mathcal W^n\cap\mathcal V^n=\emptyset$. The sets $\calV^n$ and $\calW^n$ fit into the following distance estimate.  Recall that $\Rel_{D}(x,y)$ denotes the set of all $U\in \frak S$ with $d_U(x,y)\geq D$; see Definition~\ref{defn:relevant_domains}.

\begin{lemma}\label{lem:distance-estimate}
For all $n\in\Z$, if $V\in\Rel_{5E}(x,g^nx)$, then either $V\bot U_i$ for all $i$, or $V\nest U_i$ for some $i$. Consequently, there exists a constant $A$ independent of $n$ such that 
\[
\frac1A\dist_G(x,g^{n}x)-A \;\le\; 
\underset{V \in \calV^n}{\sum} \dist_V(x,g^{n}x) \,+\!\!\!\underset{\,W \in \calW^n}{\sum} \!\dist_W(x,g^{n}x)
\;\le\; A\dist_G(x,g^nx)+A.
\]
\end{lemma}

\begin{proof}
Fix $n\in\Z$. If $V \in \mathfrak S$ satisfies $U_i \pnest V$ or $U_i \pitchfork V$ for some $i$, then $\dist_V(x,g^{s}x)\leq 3E$ for all $s \in \mathbb{N}.$ To see this, note that $\rho^{g^{s}U_i}_{g^{s}V}=\rho^{U_i}_{g^{s}V}$, since  $gU_i=U_i$, and, by definition of~$x$, we have $\dist_{g^{s}V}(x,\rho^{U_i}_{g^{s}V})\leq E$. We also have $\dist_{g^{s}V}(g^{s}x, \rho^{U_i}_{g^{s}V})=\dist_V(x,\rho^{U_i}_V) \leq E$. Hence, it follows from the triangle inequality that $\dist_{g^{s}V}(g^{s}x, x)\leq 3E$, and translating by $g^{-s}$ gives the desired result. (The extra $E$ comes from the fact that $\rho^\bullet_\bullet$ are sets of diameter at most $E$.) Thus every $V\in\relevant_{5E}(x,g^nx)$ must be either nested in some $U_i$, or orthogonal to all $U_i$. In particular, $\Rel_{D_0}(x,g^nx)=\calV^n\cup\calW^n$. The second statement is given by the distance formula, Theorem~\ref{thm:DF}, with threshold $D_0\ge6E$. 
\end{proof}

\subsection{Step 4: Controlling orthogonal terms}\label{subsec:controlling-orthogonal} 
Next, we give a lower bound on the contribution to $\dist_G(x,g^nx)$ coming from elements of $\calV^n$ by finding an upper bound on the contribution to $\dist_G(x,g^nx)$ coming from elements of $\calW^n$. 

\begin{lemma}\label{lem:proportion} 
There exist $\epsilon=\epsilon(\mathfrak S)>0$ and $N=N(g,x)$ as follows. For all $n\geq N$, there exists $U_k \in \vbig(g)$ satisfying
\[
\sum_{V\in\calV^n_k}\dist_V(x,g^nx)\ge\epsilon n.
\]
\end{lemma}

\begin{proof}  
By Lemma~\ref{lem:Big_lemma}\eqref{item:big-bounded-outside}, there is a constant $D=D(\s,g,x)$ such that $\diam(\pi_V(\langle g \rangle \cdot x))<D$ for all $V \notin \vbig(g)$. Lemma~\ref{lem:Big_lemma} is stated for $x=1$, but the bound for $x=1$ yields a bound for arbitrary $x$ in terms of $\dist_G(1,x)$ and $E$, since the maps $\pi_V$ are all $(E,E)$--coarsely Lipschitz.

\begin{claim*}
There is a constant $P=P(D,E, \frak S, G)$ such that $\underset{W \in \calW^n}{\sum} d_W(x,g^{n}x)\leq PD$ for all $n\in\Z$.
\end{claim*}

\begin{claimproof} 
Let $C=\max\{5E,2D\}$. Let $P=P(C)$ be the constant from the passing-up lemma, Lemma~\ref{lem:passingup}. 
Fix $n\geq 0$. By definition, $\calW^n$ is disjoint from $\vbig(g)$, so $\dist_W(x,g^{n}x) \leq D$ for all $W\in\calW^n$. Thus, if the claim did not hold then we would have $|\calW^n|>P$. Also by definition, $\dist_W(x,g^nx)>E$ for all $W\in\calW^n$. By the passing-up lemma, this would imply the existence of some $V\in\s$ such that $V\pconest W$ for some $W\in\calW^n$, and with $\dist_V(x,g^nx)>C\ge D$. The latter property forces $V$ to lie in $\vbig(g)$, but then $W\pnest V$ and $W\bot V$, which is a contradiction.
\end{claimproof}

Proposition~\ref{prop:HHG-uniform-undistortion} provides a positive constant $K=K(G,\s)$ such that $\dist_G(x,g^{n}x)>Kn$ for all $n\ge 0$. For such $n$ we have  
\[
\frac{Kn}{A}-A \,\leq\, \underset{V \in \calV^n}{\sum} d_V(x,g^{n}x) + \underset{W \in \calW^n}{\sum} d_V(x,g^{n}x),
\]
where $A$, provided by Lemma~\ref{lem:distance-estimate}, is independent of $n$. By the claim, the latter term is bounded above by $PD$, which is independent of $n$. Let $N=\frac{2A}K(A+PD)$. We have shown that if $n\ge N$, then
\[
\underset{V \in \calV^n}{\sum} d_V(x,g^{n}x) \,\geq\, \frac{Kn}{A}-A-PD \,\ge\, \frac{Kn}{2A}.
\]
Since the $\calV^n_k$ are disjoint for fixed $n$, the conclusion holds with $\epsilon=\frac K{2Am}$, where $m=|\vbig(g)|$ is bounded by the complexity of $\s$ and $K$ and $A$ are independent of both $n$ and $g$. 
\end{proof}

For the remainder of the proof of Theorem~\ref{thm:main}, fix a domain $U=U_k$ such that the conclusion of Lemma~\ref{lem:proportion} holds for arbitrarily large $n$. Let $\mathbb N_\epsilon$ be the set of such $n$, and let $\alpha=\alpha_k$. 

\subsection{Step 5: Accumulating distance in nested domains} 

There are now two cases to consider, depending on  how the sum in Lemma~\ref{lem:proportion} is distributed over $\calV^n_k$.  In each case, we will find a uniform lower bound on $\tau_U(g)$, which will complete the proof of the theorem.

\subsubsection*{Case 1: No relevant proper nesting} \label{subsec:no_nesting} ~

If, for our chosen $U$, all of the properly nested domains are $D_0$--irrelevant for all $n\in \mathbb N_\epsilon$, then the proof of the theorem concludes by applying Lemma~\ref{lem:proportion}.

\begin{corollary}\label{cor:small-V-gone}
If $\dist_V(x,g^nx)<D_0$ for all $V\pnest U$ and every $n\in\mathbb N_\epsilon$, then $\tau_U(g)\geq\epsilon$.
\end{corollary}

\begin{proof}
For each $n\in\mathbb N_\epsilon$, we must have $\calV^n_k=\{U\}$, and Lemma~\ref{lem:proportion} then gives $\dist_U(x,g^nx)>\epsilon n$. Since $\tau_U(g)=\lim_{n\in\mathbb N_\epsilon}\dist_U(x,g^nx)/n$, we conclude that $\tau_U(g)\geq\epsilon$.
\end{proof}

Since $\epsilon=\epsilon(\s)$, this completes the proof in this case.

\subsubsection*{Case 2: Relevant proper nesting}\label{subsec:control-nested} ~

Suppose the assumption of Corollary~\ref{cor:small-V-gone} does not hold; that is, assume  there is some $n\in\mathbb N_\epsilon$ and some $V_n\pnest U$ such that $\dist_{V_n}(x,g^nx)\ge D_0>5E$. If there is more than one such $V_n$, fix a $\nest$--maximal choice. 

If $\dist_U(x,gx)>14E$, then, since $\pi_U(x)\in\alpha$, Corollary~\ref{cor:translation_bound} implies that $\tau_U(g)>E$, and the theorem is proved for the given $g$. Hence we can assume that $\dist_U(x,gx)\leq 14E$.

The intuition behind the strategy of this part of the proof is as follows.  First, we find a domain $V$ that (intuitively, though not precisely) is relevant for $x$ and any point further along the axis of $g$ in $\mathcal CU$ than $g^kx$ for some $k$: see Figure~\ref{fig:VStatement}.  The specific way we find $V$ also shows that for any $i$, the domain $g^iV$ is  relevant for $x$ and any point further along the axis than $g^{k+i}x$.  If $i$ is large enough, then there are lots of domains $g^jV$ that are relevant for the fixed pair of points $x$ and $g^{k+i}x$; in fact, \textit{most} of the domains $g^jV$ with $0< j < i$ are relevant.  The passing up lemma gives a uniform upper bound on the number of possible relevant domains that can appear before  $\dist_U(x,g^{k+i}x)$ must be uniformly large.  From this, we deduce a uniform lower bound on translation length.  Making this argument precise takes some care.

\begin{figure}
    \centering
    \def\svgwidth{4in}
    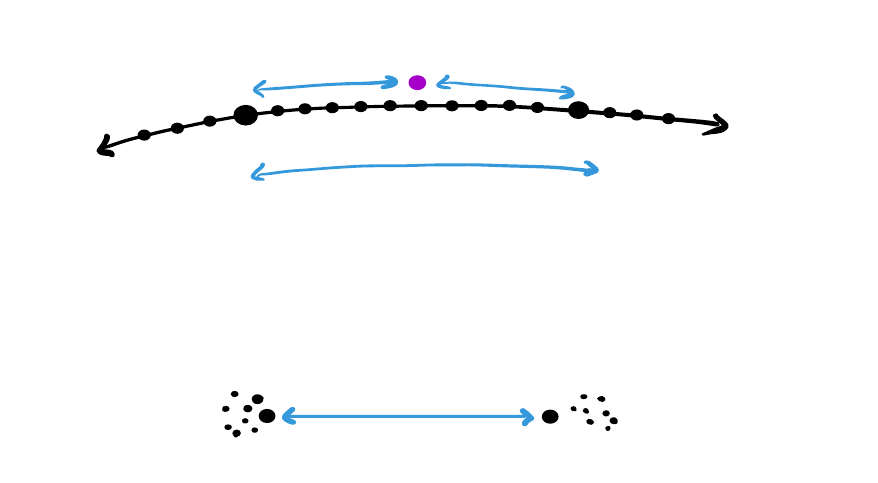
    \caption{The properties of the domain $V\pnest U$ constructed in Lemma~\ref{lem:V}.}
    \label{fig:VStatement}
\end{figure}

\begin{lemma} \label{lem:V}
Under the above assumptions, there exists $V\pnest U$ and a natural number $k$ such that the following hold.
\begin{enumerate}[label=(\roman*)]
\item   $\dist_U(x,g^kx)\le 50E$.
\item   $\dist_U(\{x,g^kx\},\rho^V_U)>5E$.
\item   If $j<0$, then $\dist_V(g^jx,x)\le E$.
\item   If $j>k$, then $\dist_V(g^kx,g^jx)\le E$.
\item   $\dist_V(x,g^kx)>3E$.
\item   If $V\pnest W\pnest U$, then $\dist_W(x,g^kx)\le7E$.
\end{enumerate}
\end{lemma}

\begin{proof}
By Lemma~\ref{lem:BGIA}, every geodesic from $\pi_U(x)$ to $\pi_U(g^nx)$ must come $E$--close to $\rho^{V_n}_U$. Since $\alpha$ is $2E$--quasiconvex, there is some point $y\in\alpha$ such that $\rho^{V_n}_U$ is contained in the $3E$--neighbourhood of $y$.

Because $x$ lies on the quasiaxis $\alpha$ of $g$, there exist (possibly negative) integers $k_0<k_1$ with $k_1-k_0$ minimal such that: $\dist_U(g^{k_i}x,y)>10E$ (and hence $\dist_U(g^{k_i}x,\rho^{V_n}_U)>7E$) and there is some $2E$--quasigeodesic from $g^{k_0}x$ to $g^{k_1}x$ that contains $y$. Intuitively, $\pi_U(g^{k_0}x)$ is the last point in $\pi_U(\langle g\rangle x)$ before the $10E$--neighbourhood of $y$ and $\pi_U(g^{k_1}x)$ is the first point in the same orbit after the $10E$--neighbourhood.

Let $V=g^{-k_0}V_n$, and let $k=k_1-k_0$.  By construction, $\dist_U(\{x,g^kx\},\rho^V_U)>7E$. This proves Item~(ii). See Figure~\ref{fig:VProof}.

\begin{figure}
    \centering
    \def\svgwidth{4in}
    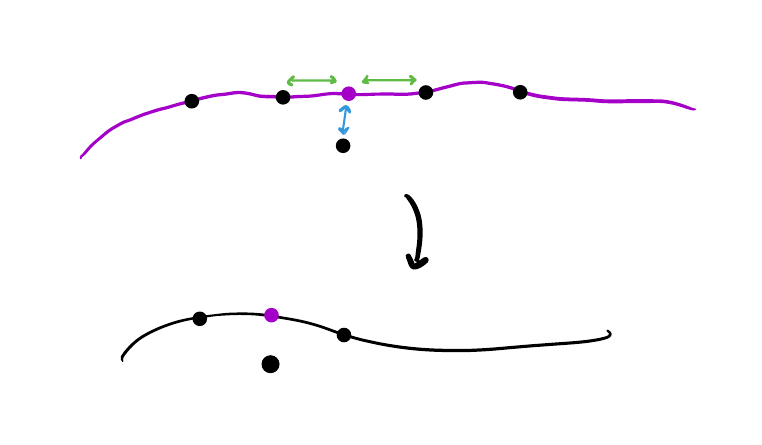
    \caption{Finding the domain $V\pnest U$.}
    \label{fig:VProof}
\end{figure}

Item (i) holds because $k_1-k_0$ was minimal and we are assuming that $\dist_U(x,gx)\le14E$:
\[
\dist_U(x,g^kx) \,\leq\, \dist_U(x,gx)+\dist_U(gx,g^{-k_0}y)+\dist_U(g^{-k_0}y,g^{k-1}x)+\dist_U(g^{k-1}x,g^kx) \,\leq\, 48E.
\]

Moreover, if $j<0$, then no geodesic from $\pi_U(g^jx)$ to $\pi_U(x)$ can come $5E$--close to $g^{-k_0}y$, and hence cannot come $E$--close to $\rho^V_U$, by the choice of $k_0$. Lemma~\ref{lem:BGIA} thus implies that $\dist_V(g^jx,x)\le E$, and so (iii) holds. Item (iv) holds for a similar reason. Together with the assumption on $V_n$, these imply (v). The final item holds by items (iii) and (iv) and our $\nest$--maximal choice of  $V_n$.  
\end{proof}

Now fix $k$ and $V\pnest U$ as in the above lemma. Let $J$ denote the minimal natural number such that $\dist_U(x,g^Jx)>400E$. Note that, although $J$ is independent of $V$, in principle it may depend on $g$. The next two lemmas show that, in fact, $J$ is bounded above independently of $g$. We note that $J\ge12$ since $\dist_U(x,gx)\leq 14E$ and $12\cdot 14E<400E$. 

We shall consider the set of all $i$ such that $g^i\rho^V_U$ is approximately half-way between $\pi_U(x)$ and $\pi_U(g^Jx)$; see Figure \ref{fig:ISchematic} for a schematic of the situation. Precisely, let 
\[
\calI \,=\, \left\{i\in\Z\,:\,\frac J3<i\le k+i<\frac{2J}3\right\}.
\]

\begin{figure}
    \centering
    \def\svgwidth{4in}
    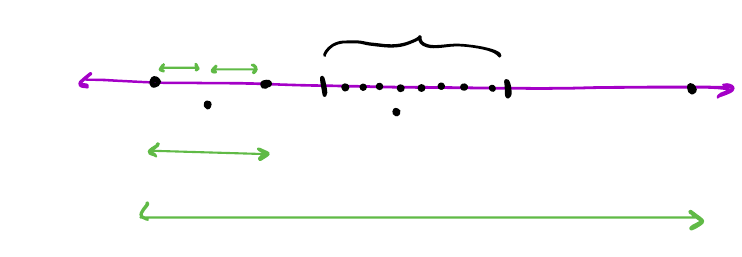
    \caption{A schematic of the sets $g^i\rho^V_U=\rho^{g^iV}_U$ when $i\in \mc I$ in $\mc CU$.}
    \label{fig:ISchematic}
\end{figure}

\begin{lemma} \label{lem:k<J}
$6k<J$, and $|\mathcal I|\geq\frac J{12}$.
\end{lemma}

\begin{proof}
By definition, $|\calI|\geq \frac{J}{3}-k-1$. Moreover, the choice of $k$ gives $\dist_U(x,g^{6k}x)\le300E$. Because $J$ is minimal with $\dist_U(x,g^Jx)> 400E$, and because $\dist_U(x,gx)\le14E$, we have $6k<J$. This shows that $|\calI|\ge\frac J6-1$, and we are done because we are assuming that $J\ge12$.
\end{proof}

Given a number $C$, let $P(C)$ be the quantity given by the passing-up lemma, Lemma~\ref{lem:passingup}.

\begin{lemma} \label{lem:I_bounded}
$|\calI|<P(500E)$.
\end{lemma}

\begin{proof}
If $i\in\calI$, then $i>0$, so by Lemma~\ref{lem:V}(iv) we have $\dist_{g^iV}(x,g^ix)=\dist_V(g^{-i}x,x)\le E$. Similarly, $J-i>k$, and so by Lemma~\ref{lem:V}(iv), $\dist_{g^iV}(g^{k+i}x,g^Jx)=\dist_V(g^kx,g^{J-i}x)\le E$. By the triangle inequality and Lemma~\ref{lem:V}(v), we therefore have 
\[
\dist_{g^iV}(x,g^Jx)>\dist_{g^iV}(g^ix, g^{k+i}x)-2E>3E-2E = E.
\]

Because $g$ acts loxodromically on $\calC U$, no power can stabilise any bounded set. In particular, $\rho^V_U$ is not stabilised by any $g^n$, and hence the $g^iV$ are pairwise distinct. Thus, if $|\calI|\ge P(500E)$, then Lemma~\ref{lem:passingup} produces a domain $W\nest U$ such that $\dist_W(x,g^Jx)>500E$ and some $g^iV$ is properly nested in $W$. 

We first argue that $W\neq U$.  Indeed, if $W=U$, then $\dist_U(x,g^Jx)>500E$, so $\dist_U(x,g^{J-1}x)>485E$, contradicting that $J$ is the minimal natural number with $\dist_U(x,g^Jx)>400E$.  Thus $W\pnest U$.

Consistency implies that $\rho^W_U$ is $E$--close to $\rho^{g^iV}_U$. Consider the domain $g^{-i}W$, into which $V$ is properly nested.  Lemma~\ref{lem:V}(iii) implies that no geodesic from $\pi_U(g^{-i}x)$ to $\pi_U(x)$ can come $E$--close to $\rho^{g^{-i}W}_U$, and hence $\dist_W(g^{-i}x,x)\le E$ by Lemma~\ref{lem:BGIA}. Also $i<\frac{2J}3-k$, so by Lemma~\ref{lem:k<J} we have $J-i>3k$. Hence Lemmas~\ref{lem:BGIA} and \ref{lem:V}(iv) similarly imply that $\dist_{g^{-i}W}(g^{J-i}x,g^Jx)\le E$. It follows from the triangle inequality that $\dist_{g^{-i}W}(x,g^Jx)>498E>7E$, which contradicts Lemma~\ref{lem:V}. 
\end{proof}

\begin{corollary} \label{cor:big-V-gone}
If the supposition of Corollary \ref{cor:small-V-gone} fails, i.e. if there is some $V\pnest U$ and some $n\in\mathbb N_\eps$ with $\dist_V(x,g^nx)\ge D_0$, then $\tau_U(g)\ge\frac E{12P(500E)}$.
\end{corollary}

\begin{proof}
By Lemmas~\ref{lem:k<J} and~\ref{lem:I_bounded}, we see that there is a number $J\le12P(500E)$ such that $\dist_U(x,g^Jx)>14E$. The result follows from Corollary~\ref{cor:translation_bound}.
\end{proof}

Since $\frac{E}{12P(500E)}$ depends only on $(G,\s)$, this completes the proof of Theorem~\ref{thm:main}.
\hfill\qedsymbol

\section{\texorpdfstring{$\s$}{S}--translation discreteness}\label{sec:strong_18}

This section discusses the sharpness of Theorem~\ref{thm:main}. The discussion is aided by the following definition. Recall from Lemma~\ref{lem:Big_lemma} that $g^{c!}U=U$ for $g\in G$ and $U\in\vbig(g)$.

\begin{definition}\label{defn:strong-18}
A hierarchically hyperbolic group $(G,\mathfrak S)$ is \emph{$\mathfrak S$--translation discrete} if there exists $\tau_0>0$ such that for all infinite-order $g\in G$, we have $\tau_U(g^{c!})\geq \tau_0$ for all $U\in\vbig(g)$. 
\end{definition}

There are two ways in which $\s$--translation discreteness is stronger than the conclusion of Theorem~\ref{thm:main}. Firstly, Theorem~\ref{thm:main} only requires $\tau_U(g^{c!})$ to be uniformly bounded away from $0$ for \emph{some} $U\in\vbig(g)$. Secondly, it does not rule out the possibility that the same $U$ supports other elements $h\in G$ with $U\in\vbig(h)$ but $\tau_U(h^{c!})$ arbitrarily small.

The following example shows that Theorem~\ref{thm:main} is sharp, by exhibiting HHG structures that are not $\s$--translation discrete.  It also shows that a group $G$ can admit HHG structures $\s_1$ and $\s_2$ such that $G$ is $\s_1$--translation discrete but not $\s_2$--translation discrete.

\begin{example} \label{exmp:modified_abelian}
Let $\Z^2=\langle a,t\mid [a,t]\rangle$.  For each $\eps\in(0,1)$, we define an HHG structure $(\Z^2,\s_\eps)$ as follows.  
\begin{itemize}
\item 	$\s_\eps=\{S,U,V\}$, where $\mathcal CS$ is a point and $\mathcal CU$ and $\mathcal CV$ are copies of $\R$. 
\item   $U\perp V$ and $U,V\nest S$.
\item   $\Z^2$ acts trivially on the set $\s_\eps$.
\item   $\pi_V:\Z^2\to\calC V$ is defined by $\pi_V(a^pt^q)=p$, for $p,q\in\Z$. This linearly extends to an action of $\Z^2$ on $\calC V$.
\item   $\pi_U:\Z^2\to\calC U$ is defined by $\pi_U(a^pt^q)=(p+q)\eps-p$.
\item   $\pi_S\colon \Z^2\to\mathcal CS$, $\rho^U_S$, and $\rho^V_S$ are defined in the only possible way.
\end{itemize}

By construction we have $\tau_U(a^pt^q)=(p+q)\eps-p$. In particular, if $\eps$ is irrational, then $\tau_U$ takes arbitrarily small positive values, so $\Z^2$ is not $\s_\eps$--translation discrete.
\end{example}

Example~\ref{exmp:modified_abelian} also shows that the constant $\tau_0$ in Theorem~\ref{thm:main} is not merely dependent on the ``skeleton'' of the HHG structure: it depends in an essential way on the parameters that bind it together. Though most of the parameters do not vary with $\eps$, the \emph{uniqueness function} does \cite[Def.~1.1(9)]{behrstockhagensisto:hierarchically:2}. One can check that this makes the constant $A$ from the distance formula (Theorem~\ref{thm:DF}) be of the order of $\frac{1}{\eps}$.  Thus, although $\tau_{\Z^2}$ is independent of $\eps$, the constant given by Lemma~\ref{lem:proportion} degenerates as $\eps\to0$, and hence, via Corollary~\ref{cor:small-V-gone}, so does $\tau_0$.

The following variation shows how Theorem~\ref{thm:main} limits the type of behaviour seen in Example~\ref{exmp:modified_abelian}.

\begin{example}
Let $\delta\in(0,1)$. Starting from $\s_\eps$, redefine $\pi_V$ by setting $\pi_V(a^pb^q)=(p+q)\delta-q$. This gives an HHG structure $\s_{\delta,\eps}$ on $\Z^2$. For generic choices of $\delta$ and $\eps$, both $\tau_U$ and $\tau_V$ can take arbitrarily small values. Although this may appear to contradict the fact that every element of $\Z^2$ has translation length at least 1, Theorem~\ref{thm:main} reassures us that no element can simultaneously realise small values of $\tau_U$ and $\tau_V$. By varying $\delta$ and $\eps$ within a fixed interval away from 0, say $(\frac12,1)$, we obtain uncountably many HHG structures $\s_{\delta,\eps}$ on $\Z^2$ that are not translation discrete, but for which the constant $\tau_0$ is the same.
\end{example}
%

Of course, $\Z^2$ \emph{is} $\s$--translation discrete with respect to its most obvious HHG structure $\s$ (the case $\eps=1$). More interesting examples will be constructed in Section~\ref{subsec:small-hhg}: we produce HHGs $(G,\s)$ that are not $\s$--translation discrete, but for which we do not know whether there exists a structure $\s'$ such that $G$ is $\s'$--translation discrete. The examples will be central extensions of HHGs.

\subsection{Positive examples} 

Here we show that some well-known HHG structures are $\s$--translation discrete.

As observed by Bowditch in \cite[Lem.~2.2]{bowditch:tight} (or by Proposition~\ref{prop:undistorted}), acylindrical actions on a hyperbolic spaces are translation discrete (positive translation lengths are uniformly bounded away from zero). Together with \cite[Thm~14.3]{behrstockhagensisto:hierarchically:1}, this shows that if $(G,\s)$ is an HHG and $S\in\mathfrak S$ is the unique $\nest$--maximal element, then $\tau_S(g)$ is uniformly bounded below for $g\in G$ satisfying $\vbig(g)=\{S\}$. This falls short of $\s$--translation discreteness, but motivates the following terminology from \cite{durhamhagensisto:boundaries}. 

\begin{definition}
An action of a group $G$ on a metric space $X$ \emph{factors through an acylindrical action} if the image of $G\to\isom X$ acts acylindrically on $X$.

We say that $(G,\s)$ is \emph{hierarchically acylindrical} if, for all $U\in\mathfrak S$, the action of $\Stab_G(U)$ on $\mathcal CU$ factors through an acylindrical action.
\end{definition}
 
In view of the above discussion, we have the following.

\begin{lemma}\label{lem:hierarchically_acylindrical_case}
If $(G,\mathfrak S)$ is hierarchically acylindrical, it is $\s$--translation discrete.
\end{lemma}

Lemma~\ref{lem:hierarchically_acylindrical_case} covers the standard HHG structure $\s$ on $G$ the fundamental groups of a compact special cube complex \cite{behrstockhagensisto:hierarchically:1}. Indeed, each $\Stab_G(U)$ is virtually a direct product of virtually compact special groups (see, e.g., \cite[Lemma 3.11]{Zalloum:effective}), one of which inherits an HHG structure where $U$ is the $\nest$--maximal element. By \cite[Thm~14.3]{behrstockhagensisto:hierarchically:1}, $(G,\mathfrak S)$ is hierarchically acylindrical.

Many examples of HHGs are not hierarchically acylindrical (even many structures on $\Z^2$; see Example~\ref{exmp:modified_abelian}), but they may still be $\s$--translation discrete. For example, if $G$ is an irreducible lattice in a product of trees, as constructed in \cite{burgermozes:lattices,wise:complete,hughes:lattices}, then the standard structure is not hierarchically acylindrical \cite{durhamhagensisto:correction}, but every $\mathcal CU$ is a tree, so loxodromic elements have combinatorial geodesic axes, so $G$ is $\s$--translation discrete. Mapping class groups also provide examples, as we now clarify.

Let $S$ be a connected, orientable, hyperbolic surface $S$ of finite type. The mapping class group $\mcg(S)$ admits a hierarchically hyperbolic structure $(\mcg(S),\mathfrak S)$, described in \cite[\S11]{behrstockhagensisto:hierarchically:2} using results in \cite{masurminsky:geometry:1,masurminsky:geometry:2,behrstockkleinerminskymosher:geometry,behrstock:asymptotic},  where $\mathfrak S$ is the set of isotopy classes of essential (not necessarily connnected) non-pants subsurfaces. For each $U\in\mathfrak S$, the associated hyperbolic space is the curve graph $\calC U$.  

When $U$ is non-annular, the action of $\Stab(U)$ on $\mathcal CU$ factors through the action of $\mcg(U)$ on $\mathcal CU$, which is acylindrical \cite[Thm~1.3]{bowditch:tight}, so translation lengths of $\mathcal CU$--loxodromic elements of $\Stab(U)$ are uniformly bounded below in terms of the topology of $U$ (and hence in terms of the topology of $S$). The same is not true when $U$ is an annulus, however, in view of the following fact. Though it is well known, we have been unable to locate a reference for it. 

The annular curve graph $\mathcal C(\gamma)$ and the action are described in \cite[\S2]{masurminsky:geometry:2}; we let $\psi:\Stab_{\mcg(S)}(\gamma)\to \isom(\mathcal C(\gamma))$ denote this action. 

\begin{proposition}\label{prop:annular-not-acyl}
Let $S$ be a connected, orientable, finite-type hyperbolic surface. Let $\gamma$ be an essential simple closed curve on $S$, and let $\mathcal C(\gamma)$ be the associated annular curve graph.  The action of $\Stab_{\mcg(S)}(\gamma)$ on $\mathcal C(\gamma)$ does not factor through an acylindrical action.
\end{proposition}

Before the proof, we have a lemma, retaining the notation from the proposition.

\begin{lemma}\label{lem:annular-kernel}
 The group $\ker\psi$ is finite.
\end{lemma}

\begin{proof}
Let $a_1,\ldots,a_k$ be a collection of non-isotopic curves such that each $a_i$ has intersection number at least one with $\gamma$, and $\{a_1,\ldots,a_k,\gamma\}$ fills $S$.  Let $p:S_1\to S$ be the annular cover associated to $\gamma$.  Let $\widetilde S\to S_1$ be the universal cover, and identify $\widetilde S$ with $\mathbb H^2$ by pulling back any hyperbolic metric on $S$.  Then the action of $\langle \gamma\rangle$ on $\widetilde S$ by deck transformations extends to an action on $\widetilde S\cup\partial\widetilde S$, and the quotient $\bar S_1$ is a closed hyperbolic annulus with interior $S_1$.  For each $i$, let $\hat a_i\colon\mathbb R\to S_1$ be an embedding whose image is a component of $p^{-1}(a_i)$ crossing the curve $\hat \gamma$ that lifts $\gamma$ and extending to a properly embedded arc $\bar a_i\to \bar S_1$ joining the two boundary circles.  So, each $\bar a_i$ represents a vertex of $\mathcal C(\gamma)$.  

Recall how each $g\in \Stab_{\mcg(S)}(\gamma)$ acts on $\bar a_i$: by the lifting criterion, $g:S\to S$ lifts to a homeomorphism $\hat g:S_1\to S_1$, which extends to a homeomorphism $\bar g:\bar S_1\to \bar S_1$, and $\bar g(\bar a_i)$ is another properly embedded arc, and it represents the vertex $\psi(g)(\bar a_i)$ of $\mathcal C(\gamma)$.  Thus far, we just summarised the construction in \cite[Sec. 2]{masurminsky:geometry:2}.

Now, if $g\in\ker\psi$, then $\psi(g)(\bar a_i)$ is isotopic rel endpoints to $\bar a_i$, which implies that $g\in\Stab_{\mcg(S)}(a_i)$.  Since this holds for all $i$, we have that $g$ stabilises each curve in a filling collection of curves, and we are done.
\end{proof}

\begin{proof}[Proof of Proposition \ref{prop:annular-not-acyl}]
For a surface $U$, write $G_U=\mcg(U)$.  Let $S_0=S-\gamma$ (we emphasise that we allow $S_0$ to be disconnected).  Let $H=\Stab_{G_S}(\gamma)$.  Then the action of $H$ on $S_0$ gives a homomorphism $\phi\colon H\to G_{S_0}$.  Let $G'_{S_0}\leq G_0$ be the finite index subgroup fixing all punctures in $S_0$, and let $H'=\phi^{-1}(G'_{S_0})$.  This gives a central extension
\[
1\to T\hookrightarrow H'\stackrel{\phi}{\longrightarrow} G_{S_0}'\to 1,
\]  
where $T$ is an infinite cyclic subgroup generated by a mapping class $\alpha$ that is a power of the Dehn twist about $\gamma$ \cite{birmanlubotzkymccarthy:abelian}.  

Let $\psi:H'\to\isom(\mathcal C(\gamma))$ be the action on the annular curve graph from \cite{masurminsky:geometry:2}.  Suppose that the action of $\psi(H')$ on $\mathcal C(\gamma)$ is acylindrical.  Then since $\psi(t)$ acts loxodromically and $\mathcal C(\gamma)$ is quasi-isometric to $\mathbb R$ (see \cite[\S2]{masurminsky:geometry:2}), the subgroup $\psi(T)$ has finite index in $\psi(G_S)$, by \cite[Lem.~6.7]{dahmaniguirardelosin:hyperbolically}.  

Let $H''=\psi^{-1}(\psi(T))$, which has finite index in $H'$. Since $\psi|_T$ is injective, the map $r\colon g\mapsto\psi|_T^{-1}(\psi(g))$ is a retraction of $H''$ onto $T$. Let $N=\ker(r)$.  Since $N\cap T$ is trivial and $T$ is central in $H''$, we have $H''=T\times N$. Also, $\phi|_N\colon N\to G_{S_0}'$ is injective and has finite-index image; in particular, $N$ is infinite.  

On the other hand, suppose that $g\in N$.  Then $\psi(g)\in\psi(T)$ since $N\leq H''$.  Moreover, $\psi(g)=\psi|_T(r(g))$ is trivial since $g\in N$.  Hence Lemma \ref{lem:annular-kernel} implies that $|N|<\infty$, a contradiction.
%
%
%
%
%
%
\end{proof}

\begin{remark}[Boundary curve variant]\label{rem:capping-homomorphism}
If $S$ is a surface with no punctures and one boundary component $\gamma$, then one can define $\mathcal C(\gamma)$ similarly to \cite{masurminsky:geometry:2}, using arcs in $S$ with at least one endpoint on $\gamma$ as vertices.  In this case, $\mcg{S}$ is a central extension of $\mcg{S'}$, where $S'$ is the corresponding punctured surface, and the central quotient comes from the \emph{capping homomorphism} (see \cite[Prop.~3.19]{farbmargalit:primer}, for instance).  In this case, a similar argument shows that the $\mcg{S}$--action on $\mathcal C(\gamma)$ does not factor through an acylindrical action: if it did, then the extension would virtually split, as in the proof of Proposition \ref{prop:annular-not-acyl}, and this would contradict that the Euler class has infinite order in $H^2(\mcg{S'},\mathbb Z)$ by \cite[\S5.5.6]{farbmargalit:primer}.
\end{remark}

Despite  Proposition~\ref{prop:annular-not-acyl}, which says that $(\mcg(S),\mathfrak S)$ is not hierarchically acylindrical, we still have the following. 


\begin{proposition}
With the standard HHG structure $\s$, mapping class groups are $\s$--translation discrete.
\end{proposition}

\begin{proof}[Sketch.]
As noted above, if $U$ is a non-annular subsurface of $S$, then the $\Stab_{\mcg{S}}(U)$--action on $\mathcal CU$ factors through the acylindrical action of $\mcg{U}$, and we are done.  

It therefore remains to consider the case where $U$ is an annulus, whose core curve we denote $\gamma$.  Specifically, we have to produce a constant $\tau>0$ such that if $g\in\Stab_{\mcg{S}}(\gamma)$, and $g$ acts on $\mathcal C(\gamma)$ loxodromically, then $\tau_{\mathcal C(\gamma)}(g)\geq \tau$. 

We now sketch an argument due to Sam Nead and Lee Mosher\footnote{See \url{https://mathoverflow.net/questions/439665/translation-length-on-annular-curve-graphs}.}; we are also grateful to Juan Souto for explaining some details to us.

First, by passing to a positive power bounded in terms of the complexity of $S$ only, we can assume that $g$ stabilises each oriented subsurface of $S-\gamma$ arising as a component of the complement of the canonical reducing system for $g$.  

First suppose that $g$ is the product of powers of Dehn twists about its reducing curves (including $\gamma)$.  In this case, one argues that there exists $B$, depending only on $S$, such that for all Dehn twists $h$ about curves disjoint from $\gamma$, and all $x\in \mathcal C(\gamma)$, we have $\dist_{\mathcal C(\gamma)}(x,hx)\leq B$.  If $t$ is the twist about $\gamma$, we thus have $g=t^kh$, where $h$ is the product of powers of twists about a uniformly bounded number of disjoint curves and $k\in\mathbb Z$, so $\dist_{\mathcal C(\gamma)}(x,g^nx)\geq \dist_{\mathcal C(\gamma)}(x,t^{kn}x)-B'$, where $B'$ is independent of $n$, and hence $\tau_{\mathcal C(\gamma)}(g)=\tau_{\mathcal C(\gamma)}(t)^k$, and we are done.

The remaining case is where $g$ acts as a pseudo-Anosov on at least one component of the complement of its reducing curves.  On each such complementary component, there is a stable train track for the restriction of $g$ to the given component, and the number of switches is bounded in terms of the complexity of $S$ (by an Euler characteristic argument; see for instance \cite[Lem. 3.2.1]{Mosher:train}), so by replacing $g$ with a uniform positive power, we can assume that $g$ fixes all of the switches.  Now we have a neighbourhood of $\gamma$ in $S$ which is a subsurface with some cusps on its boundary, and this surface is preserved by $g$, with the action fixing the cusps.  Now choose $x\in\mathcal C(\gamma)$ to be a vertex represented by an arc intersecting $\gamma$ and terminating on each side of $\gamma$ at a boundary cusp of this neighbourhood.  Thus $gx$ is another such arc with the same endpoints, and therefore isotopic rel endpoints to an arc representing a vertex of $\mathcal C(\gamma)$ in the orbit $\langle t\rangle \cdot x$.  Hence $g$ acts on $x$ like a power of the Dehn twist $t$, whose translation length thus bounds that of $g$ from below.
\end{proof}

\subsection{Quasimorphisms, central extensions, and bounded classes}\label{subsec:central-extension-notation} 

Here we recall some facts needed for the construction of HHG structures that are not $\s$--translation discrete in Section~\ref{subsec:small-hhg}. We refer the reader to \cite[Ch.~IV.3]{brown:cohomology} for more detailed background on central extensions, and \cite{calegari:scl,frigerio:bounded} for quasimorphisms.

Let $\Gamma$ be a group, and let $R\in\{\Z,\R\}$. A \emph{quasimorphism} is a map $q\colon \Gamma\to R$ such that there exists $D<\infty$ for which 
$$|q(g)+q(h)-q(gh)|\leq D$$
for all $g,h\in\Gamma$.  The infimal $D$ for which this holds is the \emph{defect} of $q$, denoted $D(q)$.  A quasimorphism $q$ is \emph{homogeneous} if $q(g^n)=nq(g)$ for all $g\in\Gamma$ and $n\in\mathbb Z$.  Given any quasimorphism $q$, the \emph{homogenisation} $\hat q\colon \Gamma\to\mathbb R$ is the homogeneous quasimorphism given by
\[
\hat q(g)=\lim_{n\to\infty}\frac{q(g^n)}{n},
\]
which has defect at most $2D(q)$.

For a group $G$, we consider central extensions 
\[
1\to \mathbb Z\to E\stackrel{\phi}{\longrightarrow} G\to 1.
\]
We always use $t$ to denote a generator of the kernel of $\phi$. The group $E$ is determined up to isomorphism by a cohomology class $[\alpha]\in H^2(G,\mathbb Z)$ (viewing $\mathbb Z$ as a trivial $\mathbb ZG$--module). More precisely, letting the $2$--cocycle $\alpha\colon G^2\to\Z$ represent $[\alpha]$, there is an isomorphism $\psi_\alpha\colon E\to E_\alpha$, where $E_\alpha$ has underlying set $G\times \mathbb Z$ and group operation $*_\alpha$ given by
\[
(g,p)*_\alpha(h,q)=(gh,p+q+\alpha(g,h));
\]
see, e.g., \cite[p. 91--92]{brown:cohomology}.  We always assume $\alpha$ is \emph{normalised}, i.e., $\alpha(g,1)=\alpha(1,g)=0$ for all $g\in G$,  which is used implicitly in defining $*_\alpha$ but not needed later.

The extension $E$ is said to \emph{arise from a bounded class} if we can moreover take $\alpha$ to be bounded as a function to $\Z$. In this case, $E$ is quasi-isometric to $G\times\R$ \cite[Thm~3.1]{gersten:bounded}. We are interested in certain quasimorphisms on such $E$. 

First, consider the map $q_\alpha=\eta\psi_\alpha\colon E\to\mathbb Z$, where $\eta\colon E_\alpha\to\mathbb Z$ is the natural projection to the second factor.  (We emphasise that $q_\alpha$ is just a map of the underlying sets, not a homomorphism.) As observed in \cite[Lem.~4.1]{hagenrussellsistospriano:equivariant} and \cite[Lem.~4.3]{hagenmartinsisto:extra}, $q_\alpha$ is a quasimorphism with $q_\alpha(t^n)=n$ for all $n\in\Z$ (perhaps after inverting $t$), by boundedness of $\alpha$.

Let $\hat q_\alpha\colon E\to\R$ be the homogenisation of $q_\alpha$.  For each infinite-order $g\in G$, the subgroup $P_g=\phi^{-1}(\langle g\rangle)$ is isomorphic to $\Z^2$ and contains $t$.  The quasimorphism $\hat q_\alpha$ restricts to a homogeneous quasimorphism $\hat q_\alpha\colon P_g\to \R$, and, since $P_g$ is abelian, $\hat q_\alpha|_{P_g}$  is a homomorphism by \cite[Prop.~2.65]{calegari:scl}.  We will use this homomorphism to choose an element $\bar g\in P_g$ and a constant $\kappa_g\in\Z$, which will be useful in the next section.

The rank of $\ker(\hat q_\alpha|_{P_g})$ is $0$ or $1$.  Suppose the kernel is nontrivial and choose a generator $\bar g$ of $\ker(\hat q_\alpha|_{P_g})$.  Hence there is a unique pair of integers $\kappa_g,\theta_g$ such that $\psi_\alpha(\bar g)=(g^{\kappa_g},\theta_g).$  Since $\hat q_\alpha(t)=1$, we have $\kappa_g\neq 0$. On the other hand, if $\hat q_\alpha\colon P_g\to\R$ is injective, choose $\bar g\in P_g-\langle t\rangle$ arbitrarily and let $\kappa_g=0$.

For any homogeneous quasimorphism $\hat p\colon G\to \mathbb R$ on $G$, the map $\hat p\phi\colon E\to \mathbb R$ is a homogeneous quasimorphism with $\hat p\phi(t)=0$. Hence 
$\hat r = \hat q_\alpha + \hat p\phi$
is a homogeneous quasimorphism on $E$, and $\hat r(t)=1$ since $\hat p\phi(t)=\hat p(1)=0$ by homogeneity of $\hat p$.

\subsection{Quasimorphisms taking arbitrarily small values} \label{subsec:brooks} 

We present two constructions of quasimorphisms on $\Z$--central extensions of groups. The first is simpler, whereas the second, which is similar to that in \cite{bestvinabrombergfujiwara:verbal}, yields more information.  There are various other constructions that could be used instead.

\subsubsection{Generalising Example~\ref{exmp:modified_abelian}} \label{subsubsec:dense-image-in-R} ~

Let $G$ be an arbitrary group admitting a nontrivial homogeneous quasimorphism $\hat p\colon G\to\R$.  Fix $g\in G$ with $\hat p(g)\neq 0$.  By homogeneity, $g$ must have infinite order, and by rescaling we can assume that $\hat p(g)=1$.  

Let $\phi\colon E\to G$ be a $\Z$--central extension arising from a bounded cocycle $\alpha$, and let $q_\alpha\colon E\to \Z$ be the quasimorphism $q_\alpha=\eta\psi_\alpha$ considered in Section~\ref{subsec:central-extension-notation}. Let $\hat q_\alpha$ be its homogenisation, and, given $\delta,\epsilon\geq0$, let $\hat r=\delta\hat q_\alpha+\epsilon\hat p\phi.$  Recall that $\hat r$ is a homomorphism on $P_g\cong\Z^2$, and note that $\hat r(t)=\delta \hat q_\alpha(t)=\delta$.  If $\hat q_\alpha$ is non-injective on $P_g$, then, in the notation of Section~\ref{subsec:central-extension-notation}, we have $\hat r(\bar g)=\epsilon\kappa_g$.  

Hence, if, for instance, $(\delta,\epsilon)=(1,\sqrt{2})$, then the map $\hat r$ takes arbitrarily small positive values on $P_g$ and therefore on $E$.  If $\hat q_\alpha$ is injective on $P_g$, then we can take $\delta=1,\eps=0$.  In this case, $\hat r(t)=1$, so by injectivity, $\hat r(\bar g)$ is irrational for the choice of $\bar g$ above, and thus $\hat r(P_g)$ is dense, so $\hat r$ takes arbitrarily small positive values on $E$.

\subsubsection{Combinations of Brooks quasimorphisms}\label{subsubsec:brooks} ~

Let $G$ be a finitely generated group admitting a nonelementary acylindrical action on a hyperbolic geodesic metric space. By \cite[Thm~6.14]{dahmaniguirardelosin:hyperbolically}, there exist $a,b\in G$ such that: $\langle a,b\rangle=F$ is a free group; $G$ has a maximal finite normal subgroup $N$; we have $\langle N,a,b\rangle\cong N\times F$; and $N\times F$ is \emph{hyperbolically embedded} in $G$.

Given a reduced, cyclically reduced word $w\in F$, define $\#_w\colon F\to\mathbb R$ by letting $\#_w(x)$ be the maximum cardinality of a set of disjoint subwords of $x$, each of which is equal to $w$. The \emph{small Brooks quasimorphism} $h_w\colon F\to\Z$ is given by $h_w(x)=\#_w(x)-\#_{w^{-1}}(x)$ \cite{brooks:some}.  By \cite[Prop.~2.30]{calegari:scl}, $h_w$ is a quasimorphism with defect at most $2$.

Define $g_i=(a^ib^i)^{101}$. This concrete choice is somewhat arbitrary, but satisfies certain small-cancellation conditions, as in \cite{bowditch:continuously,thomasvelickovic:asymptotic}. Observe that $g_i$ is not a subword of $g_j^{\pm n}$ if $j\neq i$, nor is it a subword of $g_i^{-n}$. This shows that the corresponding small Brooks quasimorphisms satisfy $h_{g_i}(g_i^n)=n$ and $h_{g_i}(g_j^n)=0$ for all $j\ne i$.

Let $(\lambda_i)_{i=1}^\infty$ be a sequence of nonzero real numbers with $\sum_{i=1}^\infty|\lambda_i|<\infty$.  Define 
\[
p_F \,=\, \sum_{i=1}^\infty\lambda_ih_{g_i}.
\]
Observe that this sum is finite for all $x\in F$, because $h_{g_i}(x)=0$ if $|g_i|>|x|$. Thus, because the $h_{g_i}$ are quasimorphisms with defect at most $2$, the map $p_F$ is a quasimorphism with defect at most $2\sum_i|\lambda_i|<\infty$. The homogenisation $\hat p_F$ of $p_F$ satisfies $\hat p_F(g_i)=\lambda_i$ for all $i$; in particular, $|\hat p_F|$ takes arbitrarily small positive values.

Extend $\hat p_F$ over $N\times F$ by declaring $\hat p_F$ to vanish on $N$. Viewed as a 1--cocycle, $\hat p_F$ is \emph{antisymmetric} (by virtue of being homogeneous).  Since $N\times F$ is hyperbolically embedded in $G$,  \cite[Thm~1.4]{hullosin:induced} provides a quasimorphism $p\colon G\to\mathbb R$ such that 
\[
L \,=\, \sup_{x\in N\times F}|\hat p_F(x)-p(x)| \,<\, \infty.
\]
The homogenisation $\hat p\colon G\to\R$ satisfies $\hat p|F=\hat p_F$. In particular, $\hat p(g_i)=\lambda_i$ for all $i$, so $|\hat p|$ takes arbitrarily small positive values on $G$.

\subsubsection{Summary} ~

We can now prove the following, which will let us construct HHG structures that are not $\s$--translation discrete.

\begin{proposition}\label{prop:acyl-quasi}
Let $\phi\colon E\to G$ be a $\mathbb Z$--central extension, associated to a bounded cohomology class, of a group $G$ that admits a nontrivial homogeneous quasimorphism.  There exists a homogeneous quasimorphism $\hat r\colon E\to\mathbb R$ such that
\begin{itemize}
    \item $\hat r(t)=1$, and
    \item for all $\epsilon>0$, there exists $e\in E$ such that $\hat r(e)\in(0,\epsilon)$.
\end{itemize}
Moreover, if $G$ has a nonelementary acylindrical action on a hyperbolic space, then $\hat r$ can be chosen with $\lim_{i\to\infty}\hat r(e_i)=0$, where $(\phi(e_i)_i)_i$ is some sequence of loxodromic elements of $G$.
\end{proposition}

\begin{proof}
As explained in Section~\ref{subsubsec:dense-image-in-R}, there exists $\hat r$ satisfying the itemised properties as soon as $G$ admits a nontrivial homogeneous quasimorphism.  

If $G$ is acylindrically hyperbolic, then we can make a more specific choice of $\hat r$ as follows.  First, let $(g_i)_i$ be loxodromic elements of $G$ chosen as in Section~\ref{subsubsec:brooks}.  For each $i$, let $\kappa_{g_i}$ be the integer chosen above by considering the restriction of $\hat q_\alpha$ to $P_{g_i}$, and let $\bar g_i\in P_{g_i}$ be the associated element.  For each $i$, if $\kappa_{g_i}=0$, let $\lambda_i=0$, and otherwise let $\lambda_i=\frac{1}{2^i\kappa_{g_i}}$.  Let $\hat p\colon G\to\R$ be the resulting homogeneous quasimorphism from Section~\ref{subsubsec:brooks}.  

Now let $\hat r=\hat q_\alpha+\hat p\phi$.  As before, $\hat r(t)=1$.  Now, for each $i$ such that $\kappa_{g_i}\neq0$, we chose $\bar g_i$ such that $\hat q_\alpha(\bar g_i)=0$ and we chose $\kappa_{g_i}$ so that $\hat p\phi(\bar g_i)=\kappa_{g_i}\hat p(g_i)$.  Hence $\hat r(\bar g_i)=\frac{1}{2^i}$.

If $\kappa_{g_i}=0$, then $\hat p(g_i)=\lambda_i=0$, so $\hat p\phi$ vanishes on $P_{g_i}$, so $\hat r=\hat q_\alpha$ on $P_{g_i}$.  Also, in this case, $\hat q_\alpha$ is an injective homomorphism on $P_{g_i}$, and $\bar g_i$ was chosen outside of $\langle t\rangle$, and thus $\hat q_\alpha(\bar g_i)\not\in\mathbb Q$.  Thus, by applying powers of $t$, we can assume $0<\hat r(\bar g_i)\leq \frac{1}{2^i}.$

Observing that $\phi(\bar g_i)$ is a nonzero power of $g_i$, we are done, taking $e_i=\bar g_i$.
\end{proof}

\subsection{HHG constructions}\label{subsec:small-hhg} 

Here we construct HHG structures that are not $\s$--translation discrete. The next lemma is \cite[Lem.~4.15]{abbottbalasubramanyaosin:hyperbolic}, except that we have extracted an additional consequence of their proof.

\begin{lemma}[Quasilines from quasimorphisms]\label{lem:quasiline}
Let $\Gamma$ be a group and let $\hat s\colon \Gamma\to\mathbb R$ be a nontrivial homogeneous quasimorphism.  There exists a graph $L$, quasi-isometric to $\mathbb R$, and a vertex-transitive, isometric action of $\Gamma$ on $L$ that fixes both ends of $L$.  Moreover, there exists $K$ such that for all $g\in\Gamma$ we have
\[
\frac1K|\hat s(g)| \,\leq\, \tau_L(g) \,\leq\, K|\hat s(g)|.
\]
\end{lemma}

\begin{proof}
Fix any positive number $C_0$ such that there is some $g_0\in\Gamma$ with $|\hat s(g_0)|=C_0$.

Let $C\ge2D(\hat s)$ be such that there is some $g\in\Gamma$ with $\hat s(g)\in(0,C/2)$. According to \cite[Lem.~4.15]{abbottbalasubramanyaosin:hyperbolic}, the set $\mathcal A$ of $g\in\Gamma$ such that $|\hat s(g)|<C$ generates $\Gamma$. Let $L=\cay(\Gamma,\calA)$. As explained in \cite{abbottbalasubramanyaosin:hyperbolic}, $L$ is quasi-isometric to $\mathbb R$, and the action of $\Gamma$ fixes the ends of $L$. The proof of \cite[Lem.~4.15]{abbottbalasubramanyaosin:hyperbolic} shows that 
\[
\frac{2C|\hat s(g)|}{3} \,\leq\, \dist_{L}(1,g) \,\leq\, \frac{|\hat s(g)|}{C_0}+2,
\]
for all $g\in \Gamma$, from which the statement about translation lengths follows.
\end{proof}

\begin{remark}\label{rem:other-quasilines}
One could deduce Lemma~\ref{lem:quasiline} from \cite[Prop.~3.1]{manning:quasiactions}; we thank Alice Kerr for this observation.  A more general statement \cite[Cor.~1.1]{kleinerleeb:induced} about quasi-actions  also works.
\end{remark}

The following lemma is extracted from the proof of \cite[Cor.~4.3]{hagenrussellsistospriano:equivariant}.

\begin{lemma}\label{lem:quasi-product}
Let $\phi\colon E\to G$ be a $\mathbb Z$--central extension of a finitely generated group  $G$. Suppose that $E$ acts by isometries on a graph $L$ that is quasi-isometric to $\mathbb R$.  Suppose further that $\tau_L(t)>0$, where $t$ generates $\ker\phi$. When $G\times L$ is equipped with the $\ell^1$--metric, the diagonal action of $E$ is proper and cobounded.
\end{lemma}

\begin{proof}
Fix a base vertex $x\in L$, and let $B$ be such that $L$ is covered by the $\langle t\rangle$--translates of the ball $B_L(x,B)$. For each $g\in G$, choose $e_g\in\phi^{-1}(g)$ such that $d_L(x,e_gx)\leq B$, which is possible because $t$ generates $\ker\phi$.

\emph{Properness.}  As $t$ is loxodromic on $L$, there exists $K$ such that $d_L(x,t^nx)\geq K|n|-K$ for all $n$. Given $R\geq 0$, let $G_R=\{g\in G \,:\, \dist_G(1,g)\leq R\}$, which is finite since $G$ is finitely generated.

Suppose  $e\in E$ moves $(1,x)$ a distance at most $R$ in $G\times L$. Then $\phi(e)\in G_R$ is one of only finitely many elements. There exists $n\in\Z$ such that $e=t^ne_{\phi(e)}$. From the triangle inequality, 
\[
d_L(x,t^nx) \,\leq\, \dist_L(x,ex) + \dist_L(t^ne_{\phi(e)}x,t^nx) \,\le\, R+B.
\]
Hence $|n|\leq (R+B+K)/K$, and so there are only finitely many such elements $e\in E$.

\emph{Coboundedness.}  Given $(g,y)\in G\times L$,  there exists $n$ such that $\dist_L(t^ne_gx,y)\leq B$. Because $t\in\ker\phi$, we have $\phi(t^ne_g)=g$, so $t^ne_g$ moves $(1,x)$ within distance $B$ of $(g,y)$.
\end{proof}

Lemma~\ref{lem:quasi-product} gives hierarchically hyperbolic structures on $\Z$--central extensions.

\begin{proposition}[HHG central extensions]\label{prop:central_ext}
Let $\phi\colon E\to G$ be a $\Z$--central extension of an HHG $(G,\s)$. Suppose $E$ acts by isometries on a graph $L$ that is quasi-isometric to $\mathbb R$.  Suppose that $\tau_L(t)>0$, where $t$ generates $\ker\phi$. The group $E$ admits an HHG structure $(E,\mathfrak S_E)$ where 
\begin{itemize}
    \item $\mathfrak S_E$ contains $\mathfrak S\sqcup \{A,S_E\}$, where $S_E$ and $A$ are two distinct symbols not in $\mathfrak S$;
    \item $\mathcal CA=L$, and $\calC W$ is a point whenever $W\in\mathfrak S_E-(\mathfrak S\sqcup\{A\})$;
    \item $A\perp U$ for all $U\in\mathfrak S$;
    \item $A\sqsubsetneq S_E$ and $U\sqsubsetneq S_E$ for all $U\in\mathfrak S$.
\end{itemize}
Moreover, $E$ stabilises $A$, the induced action on $\mathcal CA$ is the given action on $L$, and $\vbig(t)=\{A\}$. 
\end{proposition}

\begin{proof}
Equip $\mathfrak S$ with all the same relations, hyperbolic spaces, $\rho^{\bullet}_{\bullet}$ projections, as in $(G,\mathfrak S)$, but define the projections from $E\to\calC U$ for $U\in\mathfrak S$ by composing the projections $\pi_U\colon G\to\mathcal CU$ with $\phi\colon E\to G$. The projection $\pi_A\colon E\to\mathcal CA$ is an orbit map $E\to L$.  The remaining projections are maps to one-point spaces. 

The other elements of $\s_E$ are added to $\s$ as follows. By \cite[Prop.~8.27]{behrstockhagensisto:hierarchically:2}, $\mathfrak S_E$ can be chosen so that $(G\times L,\mathfrak S_E)$ with the $\ell^1$--metric is an HHS and all the bullet points in the statement are satisfied.  From the explicit description of this construction in \cite[Example 2.13]{berlairobbio:refined}, the set $\mathfrak S_E$ consists of $A,S_E,\mathfrak S$, and an element $V_U$ for each $U\in\mathfrak S$ except the $\nest$--maximal element.  The group $E$ acts on $\mathfrak S$ via the $G$--action and $\phi$. We declare $E$ to act on the set of $V_U$ in the same way and  to fix $A$ and $S_E$.  Since the $G$--action on $\mathfrak S$ is cofinite, the $E$--action on $\mathfrak S_E$ is cofinite.  It is easily verified that, with the diagonal action of $E$ on $G\times L$, the equivariance conditions of an HHG are satisfied. It remains to check that the action of $E$ on $G\times L$ is proper and cobounded, but this is given by Lemma~\ref{lem:quasi-product}.
\end{proof}

Observe that the statement of Proposition~\ref{prop:central_ext} implies that $\mathcal CS_E$ is a single point, which is consistent with the usual situation for an HHG that is coarsely a nontrivial product.

We can now show that many $\Z$--central extensions of HHGs are HHGs with structures that are not $\s$--translation discrete.

\begin{theorem}\label{thm:counterexamples}
Let $(G,\mathfrak S)$ be an HHG that is not quasi-isometric to the product of two unbounded spaces.  Then every $\Z$--central extension $E\to G$ arising from a bounded class admits an HHG structure $(E,\mathfrak S_E)$ such that $E$ is not $\s_E$--translation discrete.
\end{theorem}

\begin{proof}
By \cite[Cor.~4.7, Rem.~4.8]{petytspriano:unbounded}, $G$ is either acylindrically hyperbolic or $2$--ended.  In either case, Proposition~\ref{prop:acyl-quasi} provides a homogeneous quasimorphism $\hat r\colon E\to\mathbb R$ taking arbitrarily small positive values, with $\hat r(t)=1$, where $t$ generates $\ker(\phi)$. Lemma~\ref{lem:quasiline} gives a quasiline $L$ and an isometric $E$--action on $L$ where $\tau_L$ takes arbitrarily small positive values on $E$ but $\tau_L(t)>0$.

According to Proposition~\ref{prop:central_ext}, $E$ admits an HHG structure $(E,\mathfrak S_E)$ for which there exists $A\in\mathfrak S_E$ such that $\mathcal CA=L$, $E$ fixes $A$, and the $E$--action on $\mathcal CA$ is exactly the action on $L$ given above.  In particular, $\tau_{A}(t)>0$ and $\tau_A(e)$ takes arbitrarily small positive values as $e$ varies in $E$.  By Definition~\ref{defn:strong-18}, $E$ is thus not $\s_E$--translation discrete.
\end{proof}

\begin{remark}\label{rem:counterexample-comments}
In Theorem~\ref{thm:counterexamples}, if $G$ is acylindrically hyperbolic then the stronger statement in Proposition~\ref{prop:acyl-quasi} gives an HHG structure $(E,\s)$ where the arbitrarily small translation lengths on the quasiline are witnessed by a sequence of elements $(\tilde g_i)_i$ in $E$ whose images in $G$ are loxodromic on the top-level hyperbolic space for the original HHG structure on $G$. In fact there is a great deal of flexibility in choosing these $\tilde g_i$.
\end{remark}

\section{Questions}\label{sec:questions}

As noted in the introduction, Button showed in \cite{button:aspects} that \emph{uniform} undistortion of $\mathbb Z$ subgroups implies undistortion of $\Z^n$ subgroups, though not uniform undistortion. It is natural to ask whether this can be improved in the setting of spaces with barycentres.

\begin{question}\label{question:uu-abelian}
Suppose that $G$ acts properly and coboundedly on a space with barycentres.  Are $\Z^n$ subgroups of $G$ uniformly undistorted for each $n\geq 2$?  
\end{question}

As in the cyclic case, proving such results usually requires some form of a \emph{flat-torus} theorem \cite{gromollwolf:some,lawsonyau:compact,bridsonhaefliger:metric}. The following makes this precise.

\begin{question}\label{question:flats-barycentres}
Let $G$ act uniformly properly on a metric space $X$ with barycentres and let $n>1$. Does there exist $\lambda$ such that for all subgroups $H\leq G$ with $H\cong \Z^n$, there is an $H$--invariant subspace $F\subseteq X$ that is $(\lambda,\lambda)$--quasi-isometric to $\R^n$?
\end{question}

As discussed at the end of Section~\ref{sec:UU}, partial flat-torus statements are given by Proposition~\ref{prop:uniformly_flat_tori} for spaces with barycentres, and by work of Descombes--Lang on groups acting properly cocompactly on spaces with convex, consistent bicombings \cite{descombeslang:flats}. It seems plausible that the answer to Question~\ref{question:uu-abelian}, and hence to Question~\ref{question:flats-barycentres}, is negative in the given generality---a counterexample would be very interesting.

As mentioned in the introduction, if $G$ acts properly and cocompactly on a CAT(0) cube complex $X$, then $G$ has a WPD action on the contact graph $\calC X$ of $X$, but there are examples for which the action fails to be acylindrical because it is not translation discrete \cite{shepherd:cubulation,genevois:median}. This raises a natural question.

\begin{question}\label{question:contact_graph}
Let $G$ act properly and cocompactly on the CAT(0) cube complex $X$, and suppose that $\tau_{\mathcal CX}$ is bounded away from $0$ on loxodromic elements of $G$.  Is the induced action on $\calC X$ acylindrical?
\end{question}

One can ask analogous question for \emph{quasimedian graphs}, studied in \cite{genevois:translation}, or for the \emph{curtain models} of CAT(0) groups studied in \cite{petytsprianozalloum:hyperbolic}.

To find uniform quasi-axes for elements of an HHG $(G,\s)$, we used the $G$--action on an injective space.  This is similar to how bicombings on injective spaces are used in \cite{haettelhodapetyt:coarse} to produce equivariant bicombings on $G$.  As noted in that paper, it is unknown whether those bicombing quasigeodesics are \emph{hierarchy paths} for the given HHG structure, and we can ask the same here.

\begin{question}\label{question:hierarchy-axes} 
Given an HHG $(G,\mathfrak S)$, does there exist $D$ such that every infinite-order $g\in G$ has a $D$--quasi-axis that projects to an unparametrised $(D,D)$--quasigeodesic in every $\mathcal CU$?    
\end{question}

Note that a positive answer to Question~\ref{question:hierarchy-axes} is not at odds with the examples of Theorem~\ref{thm:intro-counterexamples} that are not $\s$--translation discrete, because it makes no requirements of the translation lengths of $g$ on the projections of its quasi-axis. Relatedly, although there are HHGs $(G,\s)$ that are not $\s$--translation discrete, Example \ref{exmp:modified_abelian} shows that there may be another HHG structure $(G,\s')$ that is $\s'$--translation discrete.

\begin{question}\label{question:all-structures}
Does there exist an HHG for which no HHG structure is $\s$--translation discrete?
\end{question}

Each hierarchically hyperbolic group $(G,\s)$ has a \textit{coarse median structure} associated with $\s$ \cite{behrstockhagensisto:hierarchically:2,bowditch:coarse}.  Distinct hierarchical structures can result in equivalent coarse median structures, but in Example~\ref{exmp:modified_abelian} the coarse median structures associated to the $\s_\eps$ are pairwise distinct.

\begin{question}\label{question:CM-s-t-d}
If $(G,\s)$ and $(G,\s')$ are HHG structures with the same associated coarse median structure, is it the case that $G$ is $\s$--translation discrete if and only if $G$ is $\s'$--translation discrete?
\end{question}

A positive answer to Question~\ref{question:CM-s-t-d} may give examples answering Question~\ref{question:all-structures}, namely irreducible lattices in products of ``bushy'' CAT$(-1)$ spaces, for example those of \cite{hughesvaliunas:commensurating}. Indeed, such lattices have unique coarse median structures by \cite{fioravantisisto:uniqueness}, and their standard structure appears to not be $\s$--translation discrete.

In the direction of finding $\s$--translation discrete structures, and in the spirit of Section~\ref{sec:strong_18}, one can ask the following.

\begin{question}\label{question:all-central-extension-HHGS}
Let $G$ be an acylindrically hyperbolic HHG, and let $\phi\colon E\to G$ be a $\Z$--central extension associated to a bounded cohomology class.  When does $E$ admit a homogeneous quasimorphism $\hat r\colon E\to\R$ such that $\hat r$ is unbounded on $\ker(\phi)$, and $|\hat r(e)|$ does not take arbitrarily small positive values as $e$ varies in $E$ (we allow $\hat r(e)=0$)?
\end{question}

Given such an $\hat r$, Proposition~\ref{prop:central_ext} produces an HHG structure $\s$ on $E$ that is $\s$--translation discrete provided $G$ admits one.  In light of \cite{abbottngsprianoguptapetyt:hierarchically}, where special consideration is given to hierarchically hyperbolic groups coarsely having a $\Z$ factor, one can also ask:

\begin{question}\label{question:coarse-direct-factor}
Let $E$ be an HHG quasi-isometric to $\Z\times A$, where $A$ is an unbounded space with an asymptotic cone that has a cut-point.  Must $E$ contain a finite-index subgroup $E'$ and a infinite-order element $t\in E'$ such that $t$ is central in $E'$ and $E'/\langle t\rangle$ is an HHG?
\end{question}

A positive answer would show that, to strengthen the results in \cite{abbottngsprianoguptapetyt:hierarchically}, $\s$--translation discreteness is most interesting for central extensions.

It may be that central extensions alone aren't enough to answer Question~\ref{question:all-structures}: there could be more elaborate examples involving complexes of groups whose vertex groups are central extensions of HHGs, assembled so that a combination theorem as in \cite{behrstockhagenmartinsisto:combinatorial} provides an HHG structure, but where the induced HHG structures on the vertex groups are forced to involve non--translation-discrete actions on quasilines. The graphs of groups in \cite{hagenrussellsistospriano:equivariant} might be a starting point.

We finish by mentioning two related avenues for research. The first concerns length spectrum rigidity. Given a group $G$ acting properly and coboundedly on an injective space $X$, one could seek conditions under which the spectrum $\tau_X(G)$ determines $X$ up to $G$--equivariant isometry, possibly among actions on injective spaces in some restricted class. One could also attempt to characterise HHG structures up to some natural equivalence from similar data. The aim would be a useful notion of the ``space of injective/HHG structures'' for $G$.  Some motivation for this idea comes from the marked $\ell^1$--length spectrum rigidity result for certain classes of actions on cube complexes \cite{beyrerfioravanti:crossratios}.  

The second considers rationality of $\tau_G$. When $G$ is hyperbolic, there is an integer $N$ such that all infinite-order elements $g\in G$ satisfy $\tau_G(g) \in \frac{1}{N} \mathbb{Z}$. The same statement holds more generally for \emph{$M$--Morse} elements of \emph{Morse local-to-global} groups \cite{russellsprianotran:local}, a class that includes groups acting properly and coboundedly on injective spaces \cite{sistozalloum:morse}. However, when $G$ is not hyperbolic, not all infinite-order elements are uniformly Morse \cite{cordesdurham:boundary}.  Thus, given a group $G$ with uniformly undistorted infinite-cyclic subgroups, one could investigate whether $\tau_G(g)$ is rational for all $g\in G$. This is unknown for mapping class groups.

\bibliography{bio}{}
\bibliographystyle{alpha}
\end{document}